\documentclass{article}
\usepackage{bbm}
\usepackage{caption}
\usepackage{rotating}
\usepackage[maxfloats=100]{morefloats}
\usepackage{listings}
\usepackage{booktabs}
\usepackage{xcolor}
\usepackage[title]{appendix}
\usepackage{makecell}
\usepackage{multirow}
\usepackage[symbol]{footmisc}
\usepackage{longtable}
\usepackage{subcaption}
\usepackage[numbers]{natbib}
\usepackage{geometry}
\usepackage{verbatim}
\usepackage{graphicx}
\usepackage{enumitem}
\usepackage{comment}
\usepackage{subfiles}
\usepackage{todonotes}
\usepackage{tikz-qtree,tikz-qtree-compat}
\usepackage{scrextend}
\usepackage{framed}
\usepackage{placeins}
\usepackage{setspace}
\usepackage[ruled,vlined]{algorithm2e}
\usepackage{mathtools,collcell,eqparbox}
\usepackage{multicol}
\usepackage{pgfplots}
\usepackage{tikz}
\usepackage{courier}
\usepackage{url}
\usepackage{color}
\usepackage{authblk}
\usepackage{amssymb}
\usepackage{amsthm}
\usepackage{dsfont}
\usepackage{amsmath}
\usepackage[pagebackref = true]{hyperref}

\hypersetup{
  colorlinks = true,
  linkcolor = teal,
  anchorcolor = teal,
  citecolor = teal,
  filecolor = teal,
  urlcolor = teal
}

\pgfplotsset{compat=1.15}
\numberwithin{equation}{section}

\makeatletter
\newcounter{BMatrix}

\newcommand{\setmaxwd}[1]{%
  \eqmakebox[BM-\theBMatrix][\BMalign]{$#1$}%
}
\MHInternalSyntaxOn

\MHInternalSyntaxOff
\makeatother

\newtheorem{theorem}{Theorem}
\newtheorem{assumption}{Assumption}

\newtheorem{proposition}{Proposition}[section]

\newtheorem{corollary}{Corollary}[theorem]
\newtheorem{lemma}{Lemma}
\theoremstyle{definition}
\newtheorem{definition}{Definition}
\newtheorem{example}{Example}
\theoremstyle{remark}
\newtheorem{remark}{Remark}
\newtheorem*{remark*}{Remark}

\DeclareMathOperator*{\esssup}{ess sup}

\DeclareMathOperator{\dist}{dist}

\newcommand{\1}{\mathds{1}}

\newcommand{\A}{\mathbb{A}}
\newcommand{\W}{\mathbb{W}}
\newcommand{\R}{\mathbb{R}}

\newcommand{\X}{\mathbb{X}}

\newcommand{\Ss}{\mathbb{S}}
\newcommand{\Z}{\mathbb{Z}}
\newcommand{\N}{\mathbb{N}}
\newcommand{\E}{\mathbb{E}}

\newcommand{\eqd}{\overset{d}{=}}

\newcommand{\sgn}{\mathrm{sign}}

\newcommand{\reg}{\mathrm{Regret}}

\newcommand{\del}{\partial}
\newcommand{\da}{\downarrow}
\newcommand{\ra}{\rightarrow}

\newcommand{\ua}{\uparrow}

\newcommand{\cd}{\cdot}
\newcommand{\ds}{\dots}

\newcommand{\mrm}[1]{\mathrm{#1}}

\newcommand{\diam}{\mathrm{diam}}

\newcommand{\KL}{\mrm{KL}}

\newcommand{\cA}{\mathcal{A}}
\newcommand{\cB}{\mathcal{B}}

\newcommand{\cD}{\mathcal{D}}

\newcommand{\cF}{\mathcal{F}}
\newcommand{\cG}{\mathcal{G}}

\newcommand{\cM}{\mathcal{M}}

\newcommand{\cP}{\mathcal{P}}

\newcommand{\cT}{\mathcal{T}}

\newcommand{\cW}{\mathcal{W}}
\newcommand{\cX}{\mathcal{X}}

\newcommand{\cZ}{\mathcal{Z}}

\newcommand{\TV}{\mathrm{TV}}

\newcommand{\set}[1]{\left\{{#1}\right\}}

\newcommand{\ceil}[1]{\left\lceil{#1}\right\rceil}

\newcommand{\norm}[1]{\left\|#1\right\|}

\newcommand{\abs}[1]{\left|#1\right|}
\newcommand{\sqbk}[1]{\left[ #1 \right]}
\newcommand{\sqbkcond}[2]{\left[ #1 \middle| #2 \right]}

\newcommand{\crbk}[1]{\left( #1 \right)}

\newcommand{\bd}[1]{\mathbf{#1}}

\newcommand{\argmax}[1]{\underset{#1}{\operatorname{arg} \operatorname{max}}\;}

\title{Fast Convergence of Policy Regret in\\ Learning Stochastic Optimal Control}
\author[1]{Shengbo Wang}
\author[2]{Jose Blanchet}
\author[2]{Peter Glynn}
\affil[1]{Daniel J. Epstein Department of Industrial and Systems Engineering\\ University of Southern California}
\affil[2]{Management Science and Engineering\\ Stanford University}
\date{May 2026}
\begin{document}
\maketitle

\begin{abstract}

Policy learning in modern operations environments faces a fundamental tension between limited operational data and the large, often continuous, state and action spaces over which good decisions must be identified and deployed. We study value-based policy learning in stochastic optimal control: a greedy policy induced by an estimate of the optimal action-value function $Q^*$ is deployed, and its performance is measured by regret. The empirical success of this approach calls for statistical insight into the structures that enable fast regret convergence. We show that, in continuous action spaces, fast policy learning is induced by three geometric structures: a growth exponent $p$, which quantifies how quickly $Q^*$ separates suboptimal actions from its maximizers; a margin-mass exponent $m$, which controls how much deployment mass lies on states with weak growth; and an action-wise regularity exponent $q$, which measures the smoothness of the $Q^*$-estimation error across actions. Given a $n^{-1/2}$-accurate estimator of $Q^*$, we show that the minimax-optimal policy regret convergence rate is
\[
    \widetilde{\Theta}\!\left(
    n^{-\min\left\{\frac{p}{2(p-q)},\,\frac{m+1}{2m}\right\}}
    \right),
\]
up to a logarithmic factor at the boundary between the two regimes. The exponent $q$ is crucial: $q>0$ yields faster-than-$n^{-1/2}$ regret. This regime is natural in operations applications. In particular, we verify $q>0$ under mild regularity conditions in dynamic inventory control and service allocation examples, while the mechanism underlying this fast rate regime extends beyond these settings.

\end{abstract}

\section{Introduction}

Stochastic control is a central language for sequential and dynamic decision-making in operations research (OR). It formalizes how controlled dynamical systems evolve over time and provides principled guidance on how a decision maker should act when inventories evolve, queues build up, demand fluctuates, or workloads move through a network. Classical stochastic control models typically take the transition dynamics, reward function, and exogenous noise distribution as known primitives, and then focus on solving the induced dynamic program \citep{bertsekas2012dynamic}. Modern operations applications are different. Model parameters, transition distributions, and sometimes even rewards or utilities must be inferred from operational data before a policy can be computed. Therefore, it becomes increasingly important to consider stochastic optimal control through a statistical lens. This perspective is now central in many modern OR applications. Dynamic pricing systems learn demand response before setting prices, inventory systems use historical demand and covariates to prescribe replenishment decisions, and service and resource allocation systems estimate uncertain network conditions before deploying routing, matching, or allocation policies \citep{besbes2009dynamic,ban2019big}. 

In these operations settings, the object ultimately inferred and deployed is a policy, and the relevant performance metric is therefore its value or regret; i.e., the performance loss of the learned policy relative to the optimal dynamic decision rule. At the same time, forecasts, parameter estimates, value estimates, and learned reward models are intermediate objects: their utility is ultimately judged by the decisions they support.

This distinction becomes particularly important in continuous state and action spaces. Many OR decisions are either inherently continuous or naturally benefit from continuous embeddings: prices, order quantities, staffing levels, production rates, allocations, and related operational decisions live more naturally in real vector spaces than as discrete labels. Such embeddings are not merely a modeling convenience. They can reveal smoothness and local structure that are invisible in worst-case discrete formulations \citep{li2020breaking}, and therefore create additional statistical leverage for policy learning. At the same time, continuous spaces make learning more delicate. The target policy is an infinite-dimensional object, while the available operational data may be very limited. These features make the way in which data are converted into a policy central to data-driven stochastic optimal control.

There are two main principles for learning good policies from data. A policy-based approach parameterizes and optimizes a policy directly, while a value-based approach estimates a surrogate value function and acts greedily. This paper studies the second route in a Markov decision process environment. We consider an estimator $\widehat Q_n$ of the optimal action-value function $Q^*$ and deploy the greedy policy $\hat \pi_n(x) \in \arg\max_{a\in \A} \widehat Q_n(x,a)$, where $x\in\X$ and $a\in\A$ represent the state and the action. This principle underlies fitted value and $Q$ iteration, projected dynamic programming, and many reinforcement learning (RL) methods \citep{bertsekas1996neuro,powell2007approximate,ernst2005tree,munos2008finite,levine2020offline}. It has strong empirical appeal: value-based and actor-critic methods often produce effective deployed policies even when global value-function estimates or Bellman errors remain far from negligible \citep{mnih2015human}. This empirical phenomenon raises a basic statistical question: \emph{Why can a greedy policy perform well before the $Q$-function is accurately estimated?}

An important insight is that a $Q$-function is used comparatively. A good greedy policy does not require the learned $Q$-function to be uniformly accurate; it requires the estimate to preserve the relevant action comparisons in the regions that matter for deployment. Estimation errors have a limited impact on performance when they still preserve nearby comparisons, occur away from good actions, or happen at states with little deployment visitation mass. Thus, the statistical difficulty of value-based policy learning is governed by how the $Q^*$ estimation error interacts with the local geometry of the action maximization problem. 

For finite action spaces, this insight has a clean form. \citet{farahmand2011action} observed that an action gap can make a greedy policy more stable: if the best action is separated from the second-best action, then moderate value-estimation errors need not change the selected action. Building on this idea, \citet{hu2025fast} developed fast policy-regret guarantees for offline RL and showed that policy regret can converge faster than the global $Q$-estimation error when the decision problem has favorable noise or margin structure. These results explain an important part of the empirical success of value-based learning. But they also reveal a limitation of the finite-action theory. In a continuous action space, an action gap is almost always absent. Indeed, if $a\ra Q^*(x,a)$ is continuous, then actions arbitrarily close to an optimizer have values arbitrarily close to optimal, hence no positive optimality gap. Therefore, a theory for the continuous action case needs a different notion of decision stability.

This paper develops such a theory. In continuous action spaces, the analogue of a finite action gap is not a discrete separation between two actions, but the local growth of the $Q$-surface near its maximizers. This growth matters in two ways: it determines how estimation error moves the greedy action, and it determines how costly the resulting action error is in value. We capture this local decision geometry through an action-growth condition.

Since we study value-based learning, the structure of the $Q$-estimation error is also important. A uniform bound on $\widehat Q_n-Q^*$ controls the maximum error gap, but greedy action selection depends more critically on how smoothly this error varies across actions. If the centered error is smooth in the action variable, then nearby actions receive similar perturbations, making the argmax more stable. This motivates an action-wise regularity condition on the estimation error, which we show is satisfied in many important OR models.

Finally, these considerations need not hold uniformly over all states. Policy regret is averaged over the states relevant for deployment, so states with weak local growth matter only to the extent that they carry mass under this distribution. The margin-mass condition controls the prevalence of such difficult states. It is the continuous-action analogue of the margin conditions in classification and finite-action policy learning \citep{mammen1999smooth,audibert2007fast,hu2025fast}, but the margin now concerns the action value growth rather than a strict action gap.

The paper gives a minimax-optimal theory of the interaction between estimation error and continuous decision geometry. At a high level, the rate is determined by three quantities. The growth exponent $p$ quantifies how quickly the true $Q$-function depreciates actions when moving away from the optimal action set. The action-wise regularity exponent $q$ describes the smoothness of the estimation error over actions. The margin exponent $m$ limits the probability mass of states where this local growth is weak. When the $Q$ estimator has the canonical $n^{-1/2}$ convergence rate, we prove matching upper and lower bounds for the policy regret rate as
\[
    \widetilde \Theta\crbk{
    n^{-\min\set{\frac{p}{2(p-q)},\frac{m+1}{2m}}}}.
\]
Here $\widetilde\Theta$ suggests minimax optimality up to logarithmic factors; the logarithmic factor appears only in the boundary case $\frac{p}{2(p-q)} = \frac{m+1}{2m}.$ The first exponent corresponds to the growth-limited regime, while the second corresponds to the margin-mass-limited regime. 

Crucially, when $q>0$ (with $q < p <\infty$ and $m<\infty$ typically holding automatically), the rate is always faster than the canonical $n^{-1/2}$. This explains how fast convergence in optimal policy learning can arise in structured OR settings. In particular, our theory applies naturally to data-driven control problems where the system dynamics are known or accessible, but the driving noise distribution must be learned from data, such as demand uncertainty in inventory and pricing, travel-time uncertainty in logistics, and service-time variability in queueing and resource allocation. Through the dynamic inventory control and service allocation examples in Section~\ref{section:verification-or-models}, we show that OR models with max-plus-linear dynamics often admit faster policy learning rates.

In addition to guaranteeing fast rates for OR settings, the optimal rate quantitatively shows how local action growth, together with smooth-in-action $Q$-estimation error, can turn a moderately accurate value estimate into a substantially more accurate greedy policy. This points to a new perspective: the design and analysis of value-based policy learning should focus on the interaction among the data-generating process, the regularity of the $Q$ estimator, and the local geometry of the decision problem.

Finally, we remark that although the paper is framed in discounted stochastic optimal control, both the upper and lower bounds cover contextual stochastic optimization \citep{kannan2025data,bertsimas2020predictive,luedtke2020performance} as a special case (with one stage). Indeed, by letting the system transition immediately to an absorbing state, the optimal value function reduces to the contextual objective $V^*(x) := \max_{a\in\A} \E[r(x,a,W)]$, where $r$ is the reward, $a\in\A$ represents the decision vector, and $x\in\X$ is the context/covariate. Then, the greedy policy against the plug-in estimate of $\E[r(x,a,W)]$ becomes the usual sample-average-approximation estimator.
Thus, the same growth, action-wise regularity, and margin-mass structure characterizes when continuous-action contextual decisions can achieve fast policy-regret rates. 

\subsection{Literature Review}

\paragraph{Data-driven stochastic optimal control and contextual stochastic optimization.}
Dynamic programming and Markov decision processes are foundational models for sequential decision-making under uncertainty \citep{puterman2014markov,bertsekas2012dynamic}. The data-driven OR literature shifts the emphasis from solving a fully specified model to learning decisions from finite operational data. This perspective appears across dynamic pricing with demand learning \citep{besbes2009dynamic}, data-driven inventory \citep{ban2019big}, renewable-energy operations \citep{kim2011optimal}, ride-sharing assignment \citep{alonso2017demand}, and dynamic treatment regimes \citep{murphy2003optimal}. Classical sample-average approximation provides a principled statistical route for stochastic optimization \citep{kleywegt2002sample,shapiro2021lectures}, while model-free methods such as $Q$-learning and policy gradient provide complementary approaches when model primitives are unavailable or difficult to estimate \citep{watkins1992q,sutton1999policy,lan2023policy,wang2026q}. Closely related one-stage formulations appear in prescriptive analytics and contextual stochastic optimization, where covariates are used to prescribe decisions rather than only predict uncertainty \citep{bertsimas2020predictive,ban2019big,kallus2023stochastic}. Predict-then-optimize and decision-focused learning make the same distinction explicit: prediction accuracy and decision quality need not coincide \citep{elmachtoub2022smart,mandi2024decision}. Our work studies the dynamic value-based analogue, where the central object is the regret of the deployed greedy policy rather than the estimation loss for the learned value function.

\paragraph{The fast convergence of policy regret phenomenon.}
Value-based methods underlie fitted value and $Q$ iterations, and offline RL \citep{bertsekas1996neuro,ernst2005tree,munos2008finite,levine2020offline}. The pioneering work of \citet{farahmand2011action} explains the fast convergence of value-based methods through the action-gap between optimal and suboptimal actions. Related gap-weighted ideas also appear in classification-based approximate policy iteration \citep{farahmand2015classification}. Most directly, \citet{hu2025fast} develops a finite-action fast-rate theory for offline reinforcement learning, showing that policy regret can converge faster than the canonical $n^{-1/2}$ rate. Related ideas appear in treatment assignment, policy learning,  and contextual bandit \citep{manski2004statistical,athey2021policy,luedtke2020performance,hu2022fast}. Our work identifies the continuous-action analogue of this action-gap phenomenon, which is inherently tied to discrete action settings. 

\paragraph{Margins, fast rates, and decision geometry.}
Fast statistical rates often arise when the induced decision loss is smoother than the estimation error. This idea goes back at least to the margin analysis of \citet{mammen1999smooth}, and was further developed through low-noise, margin, and calibration conditions in \citet{audibert2007fast} and \citet{bartlett2006convexity}. The present paper develops the analogous theory for greedy policies in continuous-action stochastic optimal control and contextual stochastic optimization.

\section{Stochastic Optimal Control and Policy Learning}
\label{section:setup}

We consider an infinite-horizon discounted stochastic optimal control problem. The state space $\X\subset\R^{d_\X}$, action space $\A\subset\R^{d_\A}$, and noise space $\W\subset\R^{d_\W}$ are endowed with their Borel $\sigma$-fields $\cX$, $\cA$, and $\cW$, respectively. We assume that $\A$ is compact and convex. Let $\Z:=\X\times\A$, equipped with the product $\sigma$-field $\cX\times\cA$. Throughout the paper, $\abs{\cd}$ denotes the Euclidean norm on arbitrary-dimensional real spaces, and $\norm{\cd}$ denotes the sup norm.

We consider a controlled stochastic system driven by a measurable state transition function $f:\X\times\A\times\W\ra\X$. The stochasticity of the system is defined on the canonical space $(\X\times\W^\N,\cX\times\cW^\N)$, supporting random variables $\set{X_0,W_1,W_2,\ds}$ such that $X_0(\omega)=x_0$ and $W_n(\omega)=w_n$ for $\omega=(x_0,w_1,w_2,\ds)\in\X\times\W^\N$. Here, $X_0$ is the initial state and $\set{W_n:n\geq 1}$ is the driving noise sequence. For an initial distribution $\mu\in\cP(\cX)$ and a common noise law $\psi\in\cP(\cW)$, we consider the probability measure $\mathbb P_\mu=\mu\times\psi^\N$ on $(\X\times\W^\N,\cX\times\cW^\N)$, where the dependence on $\psi$ is suppressed for notational convenience. 

Let $\Pi_{\mrm{HD}}$ denote the class of admissible deterministic history-dependent policies; that is, each $\pi\in\Pi_{\mrm{HD}}$ is a sequence of measurable functions $\pi=\set{\pi_t:t\geq 0}$ with $\pi_t:\X^{t+1}\to\A$. Let $\Pi_{\mrm{SD}}\subset\Pi_{\mrm{HD}}$ denote the class of measurable stationary deterministic policies. We identify each $\pi\in\Pi_{\mrm{SD}}$ with a measurable map $\pi:\X\to\A$.

For fixed $\pi\in\Pi_{\mrm{HD}}$, we write $X_0^\pi := X_0$ for notation convenience and recursively define the processes
$$
A_t^\pi:=\pi_t(X^\pi_t,X^\pi_{t-1},\ds,X_0^\pi),
\quad\text{and}\quad
X^\pi_{t+1}:=f(X^\pi_t,A_t^\pi,W_{t+1})
$$
for all $t\geq 0$. We refer to the pair $\set{(X^\pi_t,A_t^\pi):t\geq 0}$ as the controlled Markov chain under policy $\pi$. In this setting, it is convenient to define the controlled transition kernel as the pushforward measure
\begin{equation}
\label{eqn:def_P_from_f}
P(B\mid z):=\psi\crbk{\set{w\in\W:\ f(z,w)\in B}},
\end{equation}
for all $B\in\cX$ and $z\in\Z$. Thus, under $\E_\mu$, the law of $\set{X^\pi_t:t\geq 0}$ coincides with that of a stochastic process with initial distribution $\mu$ and transition distribution $P(\cd| x_t,\pi_t(x_t,\ds,x_0))$.

Let $r:\X\times\A\times\W\to\R$ be a measurable reward function. Throughout the paper, we assume that $r(z,\cdot)$ is $\psi$-integrable and define
$$
\bar r(z):=\int_\W r(z,w)\psi(dw).
$$
We then define the control performance metric as the expected infinite-horizon discounted reward
\begin{equation}\label{eqn:def_Vpi}
V^\pi(\mu):=
\E_\mu\sqbk{\sum_{t=0}^\infty \gamma^t r(X^\pi_t,A_t^\pi,W_{t+1})}. 
\end{equation}
When $\mu=\delta_x$ is a point mass at $x\in\X$, we write $\E_{\delta_x}=\E_x$ and, by abuse of notation, write $V^\pi(\delta_x)=V^\pi(x)$. Define the optimal control value function
\[
V^*(x):=\sup_{\pi\in\Pi_{\mrm{HD}}}V^\pi(x)
\]
for all $x\in\X$.

\begin{definition}[Policy regret]\label{def:regret}
For $\pi\in\Pi_{\mrm{SD}}$, define the regret from $\mu$ by
\[
\reg(\pi;\mu):=V^*(\mu)-V^\pi(\mu).
\]  
\end{definition}

The regret $\reg(\pi;\mu)$ therefore measures the loss in discounted reward incurred by using the stationary policy $\pi$, relative to the optimal value attainable starting from the initial distribution $\mu$.

When $V^*$ is measurable, we define the action-value function by
\begin{equation}
\label{eqn:def_Qstar}
Q^*(z):=
\bar r(z)+\gamma\int_{\X}V^*(x') P(dx'|z),
\qquad z=(x,a)\in\Z.
\end{equation}

In this paper, we consider a policy learning strategy based on a sequence of estimators $\{\widehat Q_n:n\geq 1\}$ that approximate $Q^*$ as $n\ra\infty$. We then take the greedy policy $\hat\pi_n$ with respect to $\widehat Q_n$ as our policy, and show that the regret can converge faster than the canonical rates under additional regularity conditions. We further prove matching lower bounds, certifying the optimality of our results.

Throughout, $\{\widehat Q_n:n\geq 1\}$ and $\{\hat\pi_n:n\geq 1\}$ are estimators defined on an underlying probability space $(\Omega,\cF,P)$, which is distinct from the canonical space introduced above. We assume that $\cF$ contains all $P$-null sets. We note that the index $n$ may, but need not, represent the number of samples used to construct the estimator $\widehat Q_n$. A rigorous abstraction is defined and analyzed in Section \ref{section:minimax-lower-bounds}. 

To ensure that the stochastic control problem is well-behaved and the optimal policy learning problem is well-posed, we first impose the following measurability requirement. To streamline the main results, we state this requirement as an abstract assumption; sufficient conditions on the model primitives $(f,r,\psi)$ are established in Section~\ref{section:suff_cond}.

\begin{assumption}[Measurability and optimality]
\label{assump:regularity}
Assume that $\bar r:\Z\to\R$, $V^*:\X\to\R$, and $Q^*:\Z\to\R$ are bounded and measurable.
For each $x\in\X$, the mapping $a\ra Q^*(x,a)$ is continuous on $\A$, and for all $x\in\X$,
\begin{equation}
\label{eqn:Vstar_max_Qstar}
V^*(x)=\max_{a\in\A}Q^*(x,a).
\end{equation}

For each $n\ge 1$ and $P$-a.e.\ $\omega\in\Omega$, the map $\widehat Q_n(\omega,\cd):\Z\to\R$ is $(\cX\times\cA)$-measurable and, for each $x\in\X$, the mapping $a\ra \widehat Q_n(\omega,x,a)$ is continuous on $\A$.
\end{assumption}

\begin{remark}\label{rmk:measurable_selection_everywhere}Since $\cF$ contains all $P$-null sets, by modifying $\widehat Q_n$ on a null set if necessary, we may and do assume throughout that the properties in Assumption \ref{assump:regularity} hold for every $\omega\in\Omega$.
\end{remark}

We first observe that, by the integrability of $r(z,\cdot)$, the boundedness of $\bar r$, and the independence of $W_{t+1}$ from $(X^\pi_t,A_t^\pi)$ under $\E_\mu$,
$$
\E_\mu\sqbk{\sum_{t=0}^\infty \gamma^t \bar r(X^\pi_t,A_t^\pi)}
=
\sum_{t=0}^\infty \gamma^t\E_\mu\sqbk{r(X^\pi_t,A_t^\pi,W_{t+1})}
=
V^\pi(\mu).
$$
Thus, we may treat $\bar r$ as the effective reward function.

On the other hand, by Assumption \ref{assump:regularity} and compactness of $\A$, the set
\begin{equation}
\label{eqn:def_Astar}
\A^*(x):=\set{a\in\A:\ Q^*(x,a)=V^*(x)}
\end{equation}
is nonempty and closed, hence compact, for every $x\in\X$.
By the measurable maximum theorem, there exists a measurable selector $\pi^*\in\Pi_{\mrm{SD}}$ such that $\pi^*(x)\in\A^*(x)$ and $V^*(x)=Q^*(x,\pi^*(x))$ for all $x\in\X$.
Moreover, the correspondence $x\to \A^*(x)$ is measurable.

Similarly, as clarified in Remark \ref{rmk:measurable_selection_everywhere}, for each $\omega\in\Omega$, we can always choose $\hat\pi_n(\omega,\cd):\X\to\A$ as a measurable selector satisfying
\begin{equation}
\label{eqn:def_hat_pi}
\hat\pi_n(\omega,x)\in \argmax{a\in\A}\widehat Q_n(\omega,x,a)
\end{equation}
for all $x\in\X$. Thus, the existence of measurable $\hat\pi_n(\omega,\cd)$ on $\Omega$ follows from the same measurable maximum theorem on this full-probability set in Assumption \ref{assump:regularity}.

We suppress the dependence of $\widehat Q_n$ and $\hat\pi_n$ on $\omega$ in the notation.
Observe that, for now, we do not assume joint measurability of $(\omega,z)\to \widehat Q_n(\omega,z)$ or $(\omega,x)\to \hat\pi_n(\omega,x)$.

\subsection{Estimation Error and Regularity in Action}

Next, we impose conditions on the statistical error of $\widehat Q_n$ that are tailored to our analysis of policy regret convergence rates. To handle measurability issues, we work with error envelopes. Specifically, the envelope $\delta_n$ controls the uniform estimation error of $\widehat Q_n$, while $\Lambda_n(q)$ controls the $q$-H\"older modulus, in the action variable, of the centered error $\widehat Q_n-Q^*$. This latter action-wise regularity is one of the main ingredients in our analysis: the greedy maximizer can be more stable when the estimation error varies smoothly across nearby actions.

\begin{assumption}[Estimation error moments bounds]
\label{assump:envelop_moment_bounds}
There exists $q\in[0,1]$ and measurable random variables $\set{\delta_n,\Lambda_n(q)>0:n\geq 1}$ such that for each $n\ge 1$ 
\[
\sup_{z\in\Z}\abs{\widehat Q_n(z)-Q^*(z)}\le \delta_n,
\]
and
\[
\sup_{x\in\X}\sup_{a\ne b}
\frac{\abs{\crbk{\widehat Q_n(x,a)-Q^*(x,a)}-\crbk{\widehat Q_n(x,b)-Q^*(x,b)}}}{\abs{a-b}^q}
\le \Lambda_n(q)
\]
a.s.$P$. Moreover, there exists a constant $C_\delta<\infty$ and, for each integer $k\ge 2$, a constant $C_{\Lambda,k}<\infty$ such that for all $n\ge 1$,
\[
E\sqbk{\delta_n^2}\le \frac{C_\delta}{n},
\qquad
E\sqbk{\Lambda_n(q)^k}\le \frac{C_{\Lambda,k}}{n^{k/2}}.
\]
\end{assumption}

\begin{remark}\label{rmk:positive_delta_Lambda_and_everywhere_bound}
We note that the requirement $\delta_n,\Lambda_n(q)>0$ is only for the convenience of future discussions, where we do not need to deal with $0\cdot +\infty$. In particular, if one finds $\delta_n',\Lambda_n'(q)\ge 0$ satisfying the rest of the requirements in Assumption \ref{assump:envelop_moment_bounds}, then
$\delta_n=\delta_n'+n^{-17}$ and $\Lambda_n(q)=\Lambda_n'(q)+n^{-17}$
satisfy Assumption \ref{assump:envelop_moment_bounds}.

As in Remark \ref{rmk:measurable_selection_everywhere}, since $\cF$ contains all $P$-null sets, after modifying $\delta_n$ and $\Lambda_n(q)$ on a null set if necessary, we assume that the a.s. bounds in Assumption \ref{assump:envelop_moment_bounds} hold for every $\omega\in\Omega$. This convention is used throughout.
\end{remark}

Assumption~\ref{assump:envelop_moment_bounds} is result-driven: the regret analysis below depends only on the envelopes $\delta_n,\Lambda_n(q)$ and their moment bounds. In Section~\ref{section:suff_cond}, we consider a data-driven stochastic optimal control setting and show that, under verifiable regularity conditions on the model primitives that usually hold in OR settings, Assumption~\ref{assump:envelop_moment_bounds} is satisfied.

\subsection{Discounted Occupancy Measures and Concentrability}

The regret $\reg(\pi;\mu)$ is a trajectory-level quantity, while the greedy policy $\hat\pi_n$ is defined via maximization of $\widehat Q_n(x,\cdot)$. To connect these two objects, we use a performance-difference identity that expresses policy regret as an occupancy-weighted Bellman suboptimality gap. We begin by defining the discounted occupancy measure. 

\begin{definition}[Discounted occupancy measure]\label{def:discounted_occupany_meas}

For a stationary deterministic policy $\pi\in\Pi_{\mrm{SD}}$ and the fixed initial distribution $\mu$, define the discounted state occupancy measure
\[
\nu_\mu^\pi(B):=(1-\gamma)\sum_{t=0}^\infty \gamma^t \mathbb P_\mu(X^\pi_t\in B),
\qquad B\in\cX.
\]
\end{definition}

The performance-difference identity below will integrate the state-wise loss $V^*(x)-Q^*(x,\pi(x))$ against $\nu_\mu^\pi$. Since this measure depends on the policy, we impose a common domination condition that allows all such integrals to be compared to integration with respect to a fixed reference measure $\lambda_\mu$.

\begin{assumption}[Concentrability relative to $\mu$]
\label{assump:concentrability}
There exist a reference probability measure $\lambda_\mu$ on $\X$ and a constant $C_{\mrm{conc}}<\infty$ such that for every stationary deterministic policy $\pi\in\Pi_{\mrm{SD}}$,
\[
\nu_\mu^\pi \ll \lambda_\mu
\qquad\text{and}\qquad
\esssup_{\lambda_\mu}\frac{d\nu_\mu^\pi}{d\lambda_\mu}\le C_{\mrm{conc}}.
\]
\end{assumption}

Intuitively, this condition holds when the discounted occupancy measures induced by stationary policies admit densities that are uniformly bounded with respect to a common reference measure. It is mild in many standard settings; for example, it is usually satisfied when $\nu_\mu^\pi$ does not concentrate at different points as $\pi$ varies. 

Next, we present a performance-difference identity in the present setting.

\begin{proposition}[Performance-difference identity]
\label{prop:performance_diff}
For every $\pi\in\Pi_{\mrm{SD}}$,
\[
\reg(\pi;\mu)
=
\frac{1}{1-\gamma}\int_{\X}\crbk{V^*(x)-Q^*(x,\pi(x))} \nu_\mu^\pi(dx).
\]
Under Assumption \ref{assump:concentrability}, this implies
\begin{equation}
\label{eqn:pdl-bound}
\reg(\pi;\mu)
\le
\frac{C_{\mrm{conc}}}{1-\gamma}\int_{\X}\crbk{V^*(x)-Q^*(x,\pi(x))} \lambda_\mu(dx).
\end{equation}
\end{proposition}
Since \eqref{eqn:def_hat_pi} holds for every $\omega\in\Omega$, \eqref{eqn:pdl-bound} applies to the greedy policy $\hat\pi_n$ everywhere.

\begin{proof}[Proof of Proposition \ref{prop:performance_diff}]
Consider the natural filtration $\mathcal{F}_t:=\sigma(X^\pi_0,\ldots,X^\pi_t)$.
By the definition \eqref{eqn:def_Qstar},
\[
Q^*(x,\pi(x))
=
\bar r(x,\pi(x))+\gamma\int_{\X}V^*(x') P(dx'|x,\pi(x)),
\]
and therefore
\[
Q^*(X^\pi_t,\pi(X^\pi_t))
=
\bar r(X^\pi_t,\pi(X^\pi_t))
+
\gamma \E_\mu\sqbk{V^*(X^\pi_{t+1})\middle|\mathcal{F}_t}
\qquad\text{a.s.}
\]
Hence, with $\ell^\pi(x):=V^*(x)-Q^*(x,\pi(x))$,
\[
\ell^\pi(X^\pi_t)
=
V^*(X^\pi_t)-\bar r(X^\pi_t,\pi(X^\pi_t))
-\gamma \E_\mu\sqbk{V^*(X^\pi_{t+1})\middle|\mathcal{F}_t}
\qquad\text{a.s.}
\]
Taking expectations, multiplying by $\gamma^t$, and summing from $t=0$ to $T$ yield
\[
\sum_{t=0}^T\gamma^t\E_\mu\sqbk{\ell^\pi(X^\pi_t)}
=
V^*(\mu)-\E_\mu\sqbk{\sum_{t=0}^T\gamma^t \bar r(X^\pi_t,\pi(X^\pi_t))}-\gamma^{T+1}\E_\mu\sqbk{V^*(X^\pi_{T+1})}.
\] 
By Assumption \ref{assump:regularity}, $V^*$ is bounded, so the last term vanishes as $T\to\infty$.
Letting $T\to\infty$ gives
\[
V^*(\mu)-V^\pi(\mu)=\sum_{t=0}^\infty \gamma^t \E_\mu\sqbk{\ell^\pi(X^\pi_t)}.
\]
Using the definition of $\nu_\mu^\pi$,
\[
(1-\gamma)\sum_{t=0}^\infty \gamma^t \E_\mu\sqbk{\ell^\pi(X^\pi_t)}
=
\int_{\X}\ell^\pi(x) \nu_\mu^\pi(dx),
\]
which proves the identity.

By \eqref{eqn:Vstar_max_Qstar},
\[
\ell^\pi(x)=V^*(x)-Q^*(x,\pi(x))\ge 0
\]
for all $x\in\X$.
Therefore Assumption \ref{assump:concentrability} yields
\[
\int_{\X}\ell^\pi d\nu_\mu^\pi
\le
C_{\mrm{conc}}\int_{\X}\ell^\pi d\lambda_\mu,
\]
which proves \eqref{eqn:pdl-bound}.
\end{proof}

\subsection{Action-Value Growth and Margin Mass}

In continuous action spaces with continuous $a\to Q^*(x,a)$, action gaps typically vanish even when the maximizer is unique. Thus, the approach used in the finite-action setting cannot work directly. Instead, the growth of $Q^*(x,\cdot)$ away from the optimal action set, together with the $\lambda_\mu$-mass of states having small growth, will determine the effect of estimation error.

\begin{assumption}[$p$-growth and margin-mass]
\label{assump:p_growth_margin_mass}
There exist $p>0$, $m>0$, a measurable function $g_p:\X\to[0,\infty)$, and a constant $M<\infty$ such that for all $x\in\X$ and $a\in\A$,
\begin{equation}
\label{eqn:Q_err_pgrowth_in_action}
V^*(x)-Q^*(x,a)\ge g_p(x) \dist(a,\A^*(x))^p,
\end{equation}
and for all $t\ge 0$,
\begin{equation}
\label{eqn:margin_mass}
\lambda_\mu(\set{x\in\X:\ g_p(x)\le t})\le M t^{1/m}.
\end{equation}
\end{assumption}

\begin{remark}\label{rmk:separate_g_p_const}
The exponent $m$ quantifies how much mass of states have a small growth coefficient (small $g_p(x)$). This is the continuous-action analogue of margin conditions in the discrete-action setting.
Note that \eqref{eqn:margin_mass} evaluated at $t=0$ implies $\lambda_\mu(\{x:\ g_p(x)=0\})=0$, so $g_p$ may vanish but only on a $\lambda_\mu$-null set.

It is convenient to separate the case when $g_p$ is uniformly bounded away from $0$. In particular, if there exists $g_\wedge>0$ such that $g_p(x)\ge g_\wedge$ for $\lambda_\mu$-a.e.\ $x$, then the analysis can be treated as a ``uniform margin'' regime (cf.\ Theorem \ref{thm:continuous_action_unif_g_ub}).
\end{remark}

Assumption~\ref{assump:p_growth_margin_mass} can be viewed as a continuous-action analogue of the margin-mass condition in classification \citep{mammen1999smooth}. In the binary classification setting, the difficult region is the neighborhood of the decision boundary, where two labels are nearly tied. In our continuous-action setting, the analogous difficulty is not an action-tie, but the set of states where the action-value surface is too flat near $\A^*$, making it statistically difficult to distinguish optimal actions from near-optimal ones. This idea is quantified by the two-layer structure of Assumption~\ref{assump:p_growth_margin_mass}.

\begin{itemize}[leftmargin=*]
    \item \textit{$p$-growth:} For a fixed $x\in\X$, if an action is at distance $d$ from the optimal action set, then the value gap measured by $Q^*$ is at least of order $g_p(x)d^p$. Thus, $p$ captures the local growth near $\A^*$ of the action-value maximization problem, while $g_p(x)$ is a state-dependent curvature coefficient. Allowing $g_p(x)$ to be small or 0 captures states at which this growth is weak and hence optimal decision-making is hard.

    \item \textit{Margin mass:} This condition controls the $\lambda_\mu$-mass where $g_p(x)$ is small. It is analogous to the Tsybakov margin condition in classification: faster rates are achieved when only a small amount of mass lies near the difficult region \citep{mammen1999smooth}.
\end{itemize}

\section{Policy Regret Upper Bounds}
\label{section:main}

We now combine the preceding ingredients to obtain upper bounds on the policy regret of the greedy policy $\hat\pi_n$. Because joint measurability of $(\omega,x)\ra \hat\pi_n(\omega,x)$ is not imposed at this stage, the results are stated in terms of measurable random variables $R_n$ that dominate the regret. When the regret itself is measurable, these bounds can be read directly as expected regret bounds.

As noted in Remark~\ref{rmk:separate_g_p_const}, we separately consider the case where $g_p$ is essentially bounded away from $0$ (Theorem \ref{thm:continuous_action_unif_g_ub}) and the case where $g_p$ satisfies a nontrivial margin-mass bound (Theorem \ref{thm:continuous_action_ub}).

\begin{theorem}
\label{thm:continuous_action_unif_g_ub}
Suppose Assumptions \ref{assump:regularity},  \ref{assump:envelop_moment_bounds}, and \ref{assump:concentrability} hold and there exist $p>q$ and a constant $g_\wedge>0$ such that
\[
V^*(x)-Q^*(x,a) \ge g_\wedge  \dist(a,\A^*(x))^p,
\]
for all $a\in\A$ and for $\lambda_\mu$-a.e.\ $x\in\X$.
Then there exists a measurable random variable $R_n$ with $R_n\ge \reg(\hat\pi_n;\mu)$ such that for all $n\geq 1$,
\[
ER_n \le C n^{-p/(2(p-q))}
\]
for some constant $C<\infty$, depending on $\gamma,C_{\mrm{conc}},p,q,g_\wedge$ and $C_{\Lambda,k_*}$, where $k_*:=\ceil{p/(p-q)}\vee 2$.
\end{theorem}

\begin{theorem}
\label{thm:continuous_action_ub}
Suppose Assumptions \ref{assump:regularity},  \ref{assump:envelop_moment_bounds}, \ref{assump:concentrability}, and \ref{assump:p_growth_margin_mass} hold. Then there exists a measurable random variable $R_n$ with
$R_n\ge \reg(\hat\pi_n;\mu)$ such that for all $n\ge 2$:
\begin{enumerate}[label=\textup{(\roman*)}]
\item if $q\leq p<q(m+1)$, then $ER_n \le C  n^{-(m+1)/(2m)}$,
\item if $p=q(m+1)$, then $ER_n \le C  n^{-(m+1)/(2m)}\log n$,
\item if $p>q(m+1)$, then $ER_n \le C  n^{-p/(2(p-q))}$,
\end{enumerate}
for some $C<\infty$, depending only on $\gamma,C_{\mrm{conc}},M,m,p,q,C_\delta$, and $\set{C_{\Lambda,k}: k\ge 2}$. 
\end{theorem}
We note that when $q< p < q(m+1)$, $\frac{m+1}{2m}< \frac{p}{2(p-q)} $; and when $p > q(m+1)$, $\frac{m+1}{2m}> \frac{p}{2(p-q)} $. Therefore, Theorem \ref{thm:continuous_action_ub} implies that for $p\ge q$, 
\begin{equation}\label{eqn:continuous_action_ub_summary}
    ER_n \le \widetilde O \crbk{n^{-\min\set{\frac{p}{2(p-q)},\frac{m+1}{2m}}}}
\end{equation}
where for the case $p = q$ we read $\frac{p}{2(p-q)} = +\infty$, and $\widetilde O$ hides a $\log n$ term only when $p = q(m+1)$.

The two exponents in \eqref{eqn:continuous_action_ub_summary} correspond to two different bottlenecks. The exponent $-p/(2(p-q))$ is the growth-limited rate. It appears when the error in estimating the optimal action from $\widehat Q_n$ is converted into regret through the $p$-growth of $Q^*(x,\cdot)$. The exponent $-(m+1)/(2m)$ is the margin-mass-limited rate. It reflects the possibility that a $\lambda_\mu$-non-negligible set of states has weak growth, so that uniformly accurate Q-estimation may still lead to unstable action selection on those states. The critical case $p=q(m+1)$ is where these two mechanisms balance, leading to the logarithmic factor.

\begin{remark}\label{rmk:q_eq_0}
    We observe that since $p>0$ in Theorem~\ref{thm:continuous_action_ub}, the case $q=0$ can only arise in (iii) when $p>q(m+1)$. Thus, when $q=0$, we recover the $O(n^{-1/2})$ rate. The same happens in Theorem \ref{thm:continuous_action_unif_g_ub}. This is the natural scaling, as the moments bounds on $\Lambda_n(0)$ do not provide additional leverage beyond that for $\delta_n$.

    On the other hand, as noted in the introduction, when $q>0$, the rate is always faster than the canonical $n^{-1/2}$. Together with the matching lower bound, this shows that the $\Lambda_n(q)$ envelope Assumption~\ref{assump:envelop_moment_bounds} plays a pivotal role in guaranteeing fast convergence. 
\end{remark}
Given the crucial role of $q$, in Section \ref{section:suff_cond} we provide a separate treatment that imposes regularity conditions on the model primitives to guarantee Assumption \ref{assump:envelop_moment_bounds}.

\subsection{A Key Property of the Greedy Policy}
 
Define the policy error at state $x\in\X$ by
\[
\ell_n(x):=V^*(x)-Q^*(x,\hat\pi_n(x)).
\]
When $p>q$, we use the convention that $g_p(x)^{-\alpha}=+\infty$ on the set $\{g_p(x)=0\}$ for any $\alpha>0$, and $g_p(x)^0\equiv 1$.

The proof of both upper bounds relies on the following key stability property of the greedy action.
\begin{proposition}[Greedy stability under $p$-growth]
\label{prop:greedy_policy_err}
Suppose Assumptions \ref{assump:regularity}, \ref{assump:envelop_moment_bounds}, and \eqref{eqn:Q_err_pgrowth_in_action} hold. For any $x\in\X$, the following hold:
\begin{enumerate}[label=\textup{(\roman*)},leftmargin=*]
\item If $p=q$, then
\[
\ell_n(x) \le 2\delta_n \1\set{g_p(x)\le \Lambda_n(q)}.
\]
\item If $p>q$, with $\alpha=q/(p-q)$ and $\beta=p/(p-q)$,
\[
\ell_n(x) \le \min\set{2\delta_n,\ \Lambda_n(q)^{\beta}  g_p(x)^{-\alpha}}.
\]
\end{enumerate}
\end{proposition}

A key insight from the proof of this proposition is that the error $\ell_n(x)$ is bounded above by the action selection error $\dist(\hat\pi_n(x),\A^*(x))^q$ and $\Lambda_n(q)$ by Assumption~\ref{assump:envelop_moment_bounds}, while it is bounded below by the action selection error raised to the power $p$ and $g_p(x)$ through \eqref{eqn:Q_err_pgrowth_in_action}. Together, these bounds imply that the action selection error is at most $\Lambda_n(q)^{1/(p-q)}g_p(x)^{-1/(p-q)}$ when $p>q$; that is, the interplay between the growth-in-action and the $Q$-estimation-error controls the action selection error at each state $x\in\X$. Substituting this estimate back into the upper bound then yields the claim in Proposition \ref{prop:greedy_policy_err}.

\begin{proof}[Proof of Proposition \ref{prop:greedy_policy_err}]
Fix $x\in\X$ and set $\hat a:=\hat\pi_n(x)$.

If $\hat a\in\A^*(x)$, then $\ell_n(x)=0$ and the bounds hold. Hence, we only need to consider the case $\hat a\notin\A^*(x)$ and set $d:=\dist(\hat a,\A^*(x))>0$.

By definition of $\hat a=\hat\pi_n(x)$ in \eqref{eqn:def_hat_pi},
$\widehat Q_n(x,\hat a)\ge \widehat Q_n(x,a^*)$ for every $a^*\in\A^*(x)$. Hence,
\[
\begin{aligned}
Q^*(x,a^*)-Q^*(x,\hat a)
&= Q^*(x,a^*)\pm \widehat Q_n(x,\hat a) \pm \widehat Q_n(x,a^*)-Q^*(x,\hat a)\\
&\le Q^*(x,a^*) + \widehat Q_n(x,\hat a) - \widehat Q_n(x,a^*)-Q^*(x,\hat a)\\
&\le \abs{\widehat Q_n(x,a^*)-Q^*(x,a^*)}+\abs{\widehat Q_n(x,\hat a)-Q^*(x,\hat a)}.
\end{aligned}
\]
Using Assumption \ref{assump:envelop_moment_bounds},
\begin{equation}\label{eqn:tu_Q_err_bd_by_delta}
Q^*(x,a^*)-Q^*(x,\hat a)\le 2\delta_n.
\end{equation}
Since $Q^*(x,a^*)=V^*(x)$ by \eqref{eqn:def_Astar}, this gives $\ell_n(x)\le 2\delta_n$.

Moreover, by Assumption \ref{assump:envelop_moment_bounds},
\[
\crbk{\widehat Q_n(x,\hat a)-Q^*(x,\hat a)}-\crbk{\widehat Q_n(x,a^*)-Q^*(x,a^*)} \le \Lambda_n(q)\abs{\hat a-a^*}^q.
\]
Using $\widehat Q_n(x,\hat a)\ge \widehat Q_n(x,a^*)$ and rearranging yields
\[
Q^*(x,a^*)-Q^*(x,\hat a)\le \Lambda_n(q)\abs{\hat a-a^*}^q.
\]
Taking the infimum over $a^*\in\A^*(x)$ gives
\begin{equation}\label{eqn:tu_Q_err_ub_by_Lambda}
\ell_n(x)\le \Lambda_n(q)\dist(\hat a,\A^*(x))^q=\Lambda_n(q) d^q.
\end{equation}
On the other hand, by \eqref{eqn:Q_err_pgrowth_in_action},
\begin{equation}\label{eqn:tu_Q_err_lb_by_growth}
\ell_n(x)=V^*(x)-Q^*(x,\hat a)\ge g_p(x)d^p.
\end{equation}

If $p=q$, then \eqref{eqn:tu_Q_err_ub_by_Lambda} and \eqref{eqn:tu_Q_err_lb_by_growth} imply $g_p(x)\le \Lambda_n(q)$. Therefore, whenever $\hat a \notin\A^*(x)$ we have $\1\{g_p(x)\le \Lambda_n(q)\}=1$, and the bound
\[
\ell_n(x)\le 2\delta_n \1\set{g_p(x)\le \Lambda_n(q)}
\]
follows from \eqref{eqn:tu_Q_err_bd_by_delta}.

If $p>q$ and $g_p(x)=0$, then when $q>0$ the convention gives $\Lambda_n(q)^\beta g_p(x)^{-\alpha}=+\infty$, so \eqref{eqn:tu_Q_err_bd_by_delta} yields the stated minimum bound; when $q=0$, we have $\alpha=0$ and $\beta=1$, so combining \eqref{eqn:tu_Q_err_bd_by_delta} and \eqref{eqn:tu_Q_err_ub_by_Lambda} gives $\ell_n(x)\le \min\set{2\delta_n,\Lambda_n(0)}$.

If $p>q$ and $g_p(x)>0$, then combining \eqref{eqn:tu_Q_err_ub_by_Lambda} and \eqref{eqn:tu_Q_err_lb_by_growth} yields $g_p(x)d^{p-q}\le \Lambda_n(q)$, i.e.
\[
d\le \crbk{\frac{\Lambda_n(q)}{g_p(x)}}^{1/(p-q)}.
\]
Substitute this into \eqref{eqn:tu_Q_err_ub_by_Lambda} to obtain
\[
\ell_n(x)\le \Lambda_n(q)^{1+q/(p-q)}g_p(x)^{-q/(p-q)}=\Lambda_n(q)^\beta g_p(x)^{-\alpha}.
\]
Taking the minimum with \eqref{eqn:tu_Q_err_bd_by_delta} yields the claim.
\end{proof}

\subsection{Proof of Theorem \ref{thm:continuous_action_unif_g_ub}}

\begin{proof}
By \eqref{eqn:pdl-bound} in Proposition \ref{prop:performance_diff},
\[
\reg(\hat\pi_n;\mu)
\le \frac{C_{\mrm{conc}}}{1-\gamma}\int_{\X}\ell_n(x) \lambda_\mu(dx).
\]
Thus it suffices to exhibit, for each $n$, a measurable random variable dominating $\int \ell_n d\lambda_\mu$ whose expectation admits the claimed rate.

Define a measurable function $g_p:\X\to[0,\infty)$ by $g_p(x)=g_\wedge$ on the $\lambda_\mu$-full-measure set where the assumed growth bound holds, and $g_p(x)=0$ otherwise. Then \eqref{eqn:Q_err_pgrowth_in_action} holds for all $x\in\X$ and $a\in\A$, while $g_p(x)\ge g_\wedge$ for $\lambda_\mu$-a.e.\ $x$.

Since $p>q$, Proposition \ref{prop:greedy_policy_err}(ii) implies
\[
\ell_n(x)\le \Lambda_n(q)^\beta g_p(x)^{-\alpha}
\qquad\text{for all }x\in\X,
\]
with $\alpha:=q/(p-q)$ and $\beta:=p/(p-q)$. Therefore,
\[
\int_{\X}\ell_n d\lambda_\mu
\le \Lambda_n(q)^\beta \int_{\X} g_p(x)^{-\alpha} \lambda_\mu(dx)
\le g_\wedge^{-\alpha}\Lambda_n(q)^\beta.
\]
Applying \eqref{eqn:pdl-bound} yields
\[
\reg(\hat\pi_n;\mu)
\le
\frac{C_{\mrm{conc}}}{1-\gamma}g_\wedge^{-\alpha}\Lambda_n(q)^\beta.
\]
Hence we may take
\[
R_n:=
\frac{C_{\mrm{conc}}}{1-\gamma}g_\wedge^{-\alpha}\Lambda_n(q)^\beta.
\]
Let $k_*:=\ceil{\beta}\vee 2$.
By monotonicity of $L^q$ norms and Assumption \ref{assump:envelop_moment_bounds},
\[
\begin{aligned}
ER_n
&=\frac{C_{\mrm{conc}}}{1-\gamma}g_\wedge^{-\alpha}E\sqbk{\Lambda_n(q)^\beta}\\
&\le \frac{C_{\mrm{conc}}}{1-\gamma}g_\wedge^{-\alpha}E\sqbk{\Lambda_n(q)^{k_*}}^{\beta/k_*}\\
&\le
Cn^{-\beta/2},
\end{aligned}
\]
for a constant $C<\infty$ depending on $\gamma,C_{\mrm{conc}},p,q,g_\wedge$, and $C_{\Lambda,k_*}$.
Since $\beta/2=p/(2(p-q))$, the result follows.
\end{proof}

\subsection{Proof of Theorem \ref{thm:continuous_action_ub}}

\begin{proof}
By \eqref{eqn:pdl-bound} in Proposition \ref{prop:performance_diff},
\begin{equation}\label{eqn:thm2_pdl_start}
\reg(\hat\pi_n;\mu)
\le \frac{C_{\mrm{conc}}}{1-\gamma}\int_{\X}\ell_n(x) \lambda_\mu(dx).
\end{equation}
As in Remark \ref{rmk:positive_delta_Lambda_and_everywhere_bound}, we work throughout with strictly positive envelopes $\delta_n,\Lambda_n(q)>0$, so all ratios used below are well defined. To simplify notation, when $p>q$ we define the exponents
\[
\alpha:=\frac{q}{p-q}, \quad\text{and}\quad \beta:=\frac{p}{p-q} = \alpha+1.
\]

We treat the cases in Theorem \ref{thm:continuous_action_ub} separately. Moreover, we split the case $q\leq p<q(m+1)$ into $p=q$ and $q<p<q(m+1)$.

\paragraph{Case $p=q$.}
By Proposition \ref{prop:greedy_policy_err}(i),
\[
\int_{\X}\ell_n(x) \lambda_\mu(dx)
\le 2\delta_n \lambda_\mu(\set{g_p\le \Lambda_n(q)}).
\]
By the margin-mass condition \eqref{eqn:margin_mass}, applied at $u=\Lambda_n(q)$,
\[
\lambda_\mu(\set{g_p\le \Lambda_n(q)})\le M\Lambda_n(q)^{1/m}.
\]
Therefore,
\[
\int_{\X}\ell_n(x) \lambda_\mu(dx)
\le 2M\delta_n\Lambda_n(q)^{1/m},
\]
and \eqref{eqn:thm2_pdl_start} yields
\[
\reg(\hat\pi_n;\mu)\le R_n,
\quad \text{with}\quad
R_n:=\frac{2MC_{\mrm{conc}}}{1-\gamma} \delta_n\Lambda_n(q)^{1/m}.
\]
Using Cauchy--Schwarz and Assumption \ref{assump:envelop_moment_bounds},
\[
E[\delta_n\Lambda_n(q)^{1/m}]
\le E[\delta_n^2]^{1/2} E[\Lambda_n(q)^{2/m}]^{1/2}.
\]
Let $k_1:=\ceil{2/m}\vee 2$.
By monotonicity of moments, $E[\Lambda_n(q)^{2/m}] \le E[\Lambda_n(q)^{k_1}]^{2/(mk_1)}$, hence
\[
ER_n
\le \frac{2MC_{\mrm{conc}}}{1-\gamma}
E[\delta_n^2]^{1/2} E[\Lambda_n(q)^{k_1}]^{1/(mk_1)}
\le C n^{-(m+1)/(2m)}.
\]

\paragraph{An auxiliary lemma.} For cases $q<p\le q(m+1)$, we first note that in this regime $\alpha\geq 1/m$. We will use the following lemma; its proof is given in Appendix \ref{section:proof:lemma:moment_truncation_bound}. 

\begin{lemma}
\label{lemma:moment_truncation_bound}
Let $G\ge 0$ be a random variable such that $\mathbb P(G\le t)\le M t^{1/m}$ for all $t\in[0,1]$.
Fix $\alpha\ge 1/m$, and let $u>0$ and $v>0$.
Interpret $G^{-\alpha}=+\infty$ on $\{G=0\}$.
\begin{enumerate}[label=\textup{(\roman*)},leftmargin=*]
\item If $\alpha>1/m$, then there exists $C_{\mrm{tr}}<\infty$, depending only on $(M,m,\alpha)$, such that
\[
E\sqbk{\min\set{2u,\ vG^{-\alpha}}}
\le C_{\mrm{tr}}\crbk{u^{1-1/(m\alpha)}v^{1/(m\alpha)} + v}.
\]
\item If $\alpha=1/m$, then
\[
E\sqbk{\min\set{2u,\ vG^{-\alpha}}}
\le (M+1)v + Mm v\max\set{0, \log\crbk{\frac{2u}{v}}}.
\]
\end{enumerate}
\end{lemma}
We note that the proof of this lemma relies on optimally truncating the law of the random variable $G$ at small values. In particular, the ``tr'' subscript in $C_{\mrm{tr}}$ stands for truncation. Equipped with Lemma \ref{lemma:moment_truncation_bound}, we analyze $q<p<q(m+1)$ and $p = q(m+1)$ separately. 
\paragraph{Case $q<p<q(m+1)$.}
Note that $\alpha>1/m$ in this regime.
By Proposition \ref{prop:greedy_policy_err}(ii),
\[
\ell_n(x)\le \min\set{2\delta_n,\ \Lambda_n(q)^\beta g_p(x)^{-\alpha}}.
\]

Let $X$ be an auxiliary random variable with law $\lambda_\mu$ under $\mathbb P$ and set $G:=g_p(X)$. By \eqref{eqn:margin_mass},
\[
\mathbb P(G\le t)=\lambda_\mu(\{g_p\le t\})\le M t^{1/m},
\]
for all $t\ge 0$. So, applying Lemma \ref{lemma:moment_truncation_bound} (i) to $G$, with $u=\delta_n$ and $v=\Lambda_n(q)^\beta$, gives
\begin{equation}\label{eqn:tu_1pm_case_err_ub}
\begin{aligned}
\int_{\X}\ell_n d\lambda_\mu &\leq  \E\sqbk{\min\set{2\delta_n,\ \Lambda_n(q)^\beta G^{-\alpha}}}\\
&\le C_{\mrm{tr}}\crbk{\delta_n^{1-1/(m\alpha)}\Lambda_n(q)^{\beta/(m\alpha)}+\Lambda_n(q)^\beta}.
\end{aligned}
\end{equation}

Since $\alpha=q/(p-q)$,
\[
1-\frac{1}{m\alpha}=\frac{q(m+1)-p}{mq},
\qquad
\frac{\beta}{m\alpha}=\frac{p}{mq}.
\]
Therefore, combining \eqref{eqn:tu_1pm_case_err_ub} and \eqref{eqn:thm2_pdl_start}, we have $\reg(\hat\pi_n;\mu)\le R_n,$ where
\begin{equation}\label{eqn:tu_Rn_for_1pm}
R_n:=
\frac{C_{\mrm{conc}}C_{\mrm{tr}}}{1-\gamma}
\crbk{\delta_n^{(q(m+1)-p)/(mq)}\Lambda_n(q)^{p/(mq)}+\Lambda_n(q)^\beta}.
\end{equation}

To control the expectation of $R_n$, set
\[
a:=\frac{q(m+1)-p}{mq}\in(0,1),
\qquad
b:=\frac{p}{mq},
\qquad
k_2:=\ceil{\max\set{2,\frac{2p}{p+q(m-1)}}}.
\]
Note that $$\frac{b}{k_2}\leq \frac{p+q(m-1)}{2mq},$$ so $$\frac{a}{2}+\frac{b}{k_2}\le \frac{q(m+1)-p}{2mq} + \frac{p+q(m-1)}{2mq} = 1.$$
Applying the generalized H\"older inequality in Lemma \ref{lemma:generalized_holder} gives
\begin{align*}
E[\delta_n^a\Lambda_n(q)^b]
&\le
E[\delta_n^2]^{a/2} E[\Lambda_n(q)^{k_2}]^{b/k_2}\\
&\le C n^{-a/2}n^{-b/2}\\
&= C n^{-(m+1)/(2m)}.
\end{align*}
where the second step follows from Assumption \ref{assump:envelop_moment_bounds} and $C$ satisfies the dependence conditions in Theorem \ref{thm:continuous_action_ub}. 

For the second term in \eqref{eqn:tu_Rn_for_1pm}, let $k_\beta:=\ceil{\beta}\vee 2$.
Then by Jensen's inequality, 
\[
E[\Lambda_n(q)^\beta]
\le E[\Lambda_n(q)^{k_\beta}]^{\beta/k_\beta}
\le C n^{-\beta/2}.
\]
with some possibly different $C$ that also satisfies the dependence conditions in Theorem \ref{thm:continuous_action_ub}. Since $q<p<q(m+1)$ implies $$\frac{\beta}{2} = \frac{p}{2(p-q)}>\frac{m+1}{2m},$$ the term $E[\Lambda_n(q)^\beta]$ is of smaller order. Hence $ER_n\le C n^{-(m+1)/(2m)}$ where the $C$ is the max of the two cases.

\paragraph{Case $p=q(m+1)$.}
In this case $\alpha=1/m$. As in the previous case, Proposition \ref{prop:greedy_policy_err} (ii) yields
\[
\ell_n(x)\le \min\set{2\delta_n,\ \Lambda_n(q)^\beta g_p(x)^{-\alpha}}.
\]
Let $X\sim\lambda_\mu$ under $\mathbb P$ and set $G=g_p(X)$. Lemma \ref{lemma:moment_truncation_bound} (ii) with $u=\delta_n$ and $v=\Lambda_n(q)^\beta$ gives
\begin{align*}
\int_{\X}\ell_n d\lambda_\mu&= \E\sqbk{\min\set{2\delta_n,\ \Lambda_n(q)^\beta G^{-\alpha}} }\\
&\le (M+1)\Lambda_n(q)^\beta + Mm \Lambda_n(q)^\beta\max\set{0,\log\crbk{\frac{2\delta_n}{\Lambda_n(q)^\beta}}}.
\end{align*}
Therefore, $\reg(\hat\pi_n;\mu)\le R_n,$ where
\begin{equation}\label{eqn:tu_Rn_case_peqm}
R_n:=
\frac{C_{\mrm{conc}}}{1-\gamma}
\sqbk{(M+1)\Lambda_n(q)^\beta + Mm \Lambda_n(q)^\beta\max\set{0,\log\crbk{\frac{2\delta_n}{\Lambda_n(q)^\beta}}}}.
\end{equation}

To bound the expectation of the logarithmic term, we use the following lemma, proved in Appendix \ref{section:proof:lemma:critical_log}.

\begin{lemma}
\label{lemma:critical_log}
Suppose Assumption \ref{assump:envelop_moment_bounds} holds.  Fix $\beta>1$ and set $k_\beta:=\ceil{\beta+1}$.
Then for all $n\ge 2$,
\[
E\sqbk{\Lambda_n(q)^\beta\max\set{0,\log\crbk{\frac{2\delta_n}{\Lambda_n(q)^\beta}}}}
\le C n^{-\beta/2}\log n,
\]
for some $C<\infty$, depending only on $(\beta,C_\delta,C_{\Lambda,k_\beta})$.
\end{lemma}

Moreover, with $k_\beta=\ceil{\beta+1}$,
\[
E[\Lambda_n(q)^\beta]
\le E[\Lambda_n(q)^{ k_\beta}]^{\beta/ k_\beta}
\le C n^{-\beta/2}.
\]
Since $\beta/2=(m+1)/(2m)$, we obtain $ER_n\le C n^{-(m+1)/(2m)}\log n$ for some possibly different $C < \infty$ that satisfies the requirement of Theorem \ref{thm:continuous_action_ub} and depends on $\set{C_{\Lambda,k}:k\geq 2}$ only through $C_{\Lambda,k_\beta}$. 

\paragraph{Case $p>q(m+1)$.}
In this case, $\alpha\in[0,1/m)$.
As in the previous cases, Proposition \ref{prop:greedy_policy_err} (ii) yields
\begin{equation}\label{eqn:tu_ln_bd_pgm}
\begin{aligned}
\ell_n(x)&\le \min\set{2\delta_n,\ \Lambda_n(q)^\beta g_p(x)^{-\alpha}}\\
&\le \Lambda_n(q)^\beta g_p(x)^{-\alpha}\1\set{g_p(x) > 0} + 2\delta_n\1\set{g_p(x) = 0}.   
\end{aligned}
\end{equation}

We will use the following lemma, whose proof is given in Appendix \ref{section:proof:lemma:inv_moment}.

\begin{lemma}
\label{lemma:inv_moment}
Let $G\ge 0$ satisfy $\mathbb P(G\le u)\le M u^{1/m}$ for all $u\in[0,1]$.
Fix $\alpha\in [0,1/m)$.
Then
\[
E\sqbk{G^{-\alpha} \1\{G>0\}}\le 1+\frac{Mm\alpha}{1-m\alpha}.
\]
\end{lemma}

As before, let $X\sim \lambda_\mu$ be an auxiliary random variable with law $\lambda_\mu$ under $\mathbb P$ and set $G:=g_p(X)$. Then $\mathbb P(G\le u)\le M u^{1/m}$ on $u\in[0,1]$ by \eqref{eqn:margin_mass}. Therefore, Lemma \ref{lemma:inv_moment} implies
\[
C_{g}:=\E\sqbk{G^{-\alpha} \1\{G>0\}}<\infty.
\]
Therefore, from \eqref{eqn:tu_ln_bd_pgm} and Assumption \ref{assump:p_growth_margin_mass}, we have
\[
\begin{aligned}
\int_{\X}\ell_n d\lambda_\mu
&\le \Lambda_n(q)^\beta\int_{\X}g_p(x)^{-\alpha}\1\set{g_p(x) > 0}  \lambda_\mu(dx) + 2\delta_n\lambda_\mu(\set{g_p =0})\\
&= \Lambda_n(q)^\beta \E\sqbk{G^{-\alpha} \1\{G>0\}} + 2\delta_n\lambda_\mu(\set{g_p \leq 0})\\
&\le C_{g}\Lambda_n(q)^\beta.
\end{aligned}
\]
This, combined with \eqref{eqn:thm2_pdl_start}, yields $\reg(\hat\pi_n;\mu)\le R_n$ where 
\[
R_n:=
\frac{C_{\mrm{conc}}C_{g}}{1-\gamma}\Lambda_n(q)^\beta.
\]
Let $k_\beta:=\ceil{\beta}\vee 2$. Then
\[
E[\Lambda_n(q)^\beta]
\le E[\Lambda_n(q)^{k_\beta}]^{\beta/k_\beta}
\le C n^{-\beta/2}
=
C n^{-p/(2(p-q))},
\]
where $C$ depends on $\set{C_{\Lambda,k}:k\ge2}$ only through $C_{\Lambda,k_\beta}$. 
This completes the proof of Theorem \ref{thm:continuous_action_ub}.
\end{proof}

\section{Minimax Lower Bounds and Optimality}\label{section:minimax-lower-bounds}

In this section, we prove matching minimax lower bounds on the expected regret for policy learning that complement the upper bounds in Theorems~\ref{thm:continuous_action_unif_g_ub} and \ref{thm:continuous_action_ub}. The lower bound is based on MDP instances that satisfy the assumptions of the upper bounds in a strong sense. 

Specifically, for any fixed $\gamma\in(0,1)$, $q\in(0,1]$, $p>q$, and $m\geq 0$, we define two families of $\theta\in\Theta=[-1,1]$-indexed MDP instances, denoted by $\set{\cM_\theta^+,\cM_\theta^-:\theta\in\Theta}$. To simplify notation, we suppress the dependence of $\cM_\theta^+$ and $\cM_\theta^-$ on $(\gamma,m,p,q)$, and use these families of instances only for fixed $(\gamma,m,p,q)$. The main result of this section can be understood as follows.

\paragraph{Informal version of Theorem \ref{thm:minimax_lb}:} We show that the following (informal) statement holds for each fixed feasible $(\gamma,m,p,q)$.
\begin{itemize}
    \item Assumptions \ref{assump:regularity}, \ref{assump:concentrability}, and \ref{assump:p_growth_margin_mass} holds for all instances in $\set{\cM_\theta^+,\cM_\theta^-:\theta\in\Theta}$ with the corresponding $(m,p)$. 
    \item There exists a data generating process (to be rigorously defined later) from which we can construct estimators $\widehat Q_n^+$ and $\widehat Q_n^-$  so that their estimation error compared to the corresponding optimal Q-functions $Q_\theta^+$ and  $Q_\theta^-$ satisfies Assumption \ref{assump:envelop_moment_bounds} with H\"older exponent $q$. 
    \item Given the data generating process, the minimax expected regret for optimal policy learning over the family $\set{\cM_\theta^+:\theta\in\Theta}$ is lower bounded by $\Omega\crbk{n^{-(m+1)/(2m)}}$, while over the family $\set{\cM_\theta^-:\theta\in\Theta}$ it is lower bounded by $\Omega\crbk{n^{-p/(2(p-q))}}$.
\end{itemize}
Combining the two cases, we obtain lower bounds that match the upper bounds in Theorems~\ref{thm:continuous_action_unif_g_ub} and \ref{thm:continuous_action_ub}, with only a $\log n$ mismatch in the case $p=q(m+1)$.

\begin{remark}
We note that the performance of a policy learning algorithm naturally depends on the quality of the data: if the data generating process reveals the optimal Q-function in one step, then one can achieve zero regret with $n=1$. Thus, the lower bound should be interpreted as follows: even if one can construct an estimator of the optimal Q-function satisfying Assumption~\ref{assump:envelop_moment_bounds}, one cannot in general guarantee a faster convergence rate than the stated lower bounds without additional assumptions on the MDP instance or the data generating process.
\end{remark}

Next, we establish the minimax lower bound in a mathematically rigorous manner by carefully specifying the hard instance families, the data-generating processes, the admissible learning algorithms, and the minimax expected regret criterion.

\subsection{Margin-Mass and Growth-Rate Limited Families of MDPs}\label{section:regret_instance}

In this section, we construct the families of MDPs $\set{\cM_\theta^+,\cM_\theta^-:\theta\in\Theta}$ that are used in the subsequent proofs of minimax lower bounds. We note that $\set{\cM_\theta^+:\theta\in\Theta}$ will correspond to the \textit{$m$-margin-mass-limited} family, hence leading to a lower bound of $\Omega\crbk{n^{-(m+1)/(2m)}}$. On the other hand $\set{\cM_\theta^-:\theta\in\Theta}$ is \textit{$p$-growth-limited} leading to a lower bound of $\Omega\crbk{n^{-p/(2(p-q))}}$. 

Specifically, all the MDP instances we construct will have common state, action, and noise spaces, transition function, and initial distribution.  The only differences are the reward structures and noise distributions. Specifically, we define the state space to be $\X := [0,1]\cup \set{2}$, where $2$ is considered a terminal absorbing state. The action and noise spaces are $\A := [-1,1]$ and $\W = \set{-1,1}\times[0,1]$. The transition function is the deterministic absorption $f(x,a,w) = 2$ for all $(x,a,w)\in \X\times\A\times \W$. The driving noise of the MDP is an i.i.d. sequence $\set{W_i:i\ge 1}$, where each $W_i = (Y_i,D_i)\in\W$. We will use a uniform initial distribution $\mu$ on $[0,1]$.

The families $\set{\cM_\theta^+:\theta\in\Theta}$ are indexed by a parameter $\theta\in\Theta \subset [-1,1]$. To reflect the common setting in which the reward is specified by the modeler, and where knowledge of the reward function alone does not allow direct inference of $\theta$, we deliberately design the reward functions so that they are independent of $\theta$. Under these constructions, the instances $\cM_\theta^+$ and $\cM_\theta^-$ can therefore be identified with the respective pairs consisting of a reward function and a law of $W$, namely $(r^+,\psi_\theta^+)$ and $(r^-,\psi_\theta^-)$, where the corresponding reward functions and probability laws are defined as follows.

\subsubsection*{The $\set{\cM_\theta^+:\theta\in\Theta}$ Family}
For $\cM_\theta^+$, we consider the function \begin{equation}\label{eqn:def_hp}
 h_p(a)
:=\frac{(1+a)^p}{(1+a)^p+(1-a)^p}.
\end{equation}
Then, the reward function $r^+$ is defined by
$$r^+(x,a,(y,d)):=x^m h_p(a)+ y\1\set{x\leq d}\crbk{1- h_p(a)},\quad r^+(2,a,(y,d)) = 0$$
for $x\in[0,1]$, $a\in \A$, and $(y,d)\in \W$. We note an important property of $h_p$ underlying this construction: $h_p(a)=\Theta((1+a)^p)$ as $a\da -1$, and $1-h_p(a)=\Theta((1-a)^p)$ as $a\ua 1$. 

In this case, the law \begin{equation}\label{eqn:psitheta+}
    \mrm{Law}_{\theta}^+(W_1) = \psi_\theta^+
: = \frac{1-\theta}{2}\delta_{(-1,d_\theta^+)} + \frac{1+\theta}{2}\delta_{(1,d_\theta^+)},\quad d_\theta^+ = \crbk{\frac{|\theta|}{2}}^{1/m}.
\end{equation} 
 Thus, under $\psi_\theta^+$, $Y = \pm1$ with probability $(1\pm \theta)/2$ and $D = d^+_\theta$ almost surely. 

Since the controlled Markov chain transitions deterministically to state 2, where the associated reward is equal to 0, it follows that the optimal Q function is
\begin{equation}
    Q_\theta^+(x,a) = \int r^+(x,a,w)\psi_\theta(dw) = x^m h_p(a) + \theta\1\set{x\leq d^+_\theta}(1-h_p(a)).\label{eqn:Qtheta+}
\end{equation}
We note that $h_p$ is designed to enforce the $p$-growth condition. The key intuition behind this construction leading to the lower bound rate is that, when $\theta$ is small and $x\leq d_\theta^+$, identifying the optimal action is difficult, since the optimum is $+1$ or $-1$ depending on the sign of $\theta$. In particular, the hard margin width in this case is $d_\theta^+\asymp |\theta|^{1/m}$, while the value gap incurred by taking the wrong action is of order $|\theta|$. Therefore, if $|\theta|\asymp n^{-1/2}$, the sign of $\theta$ cannot be identified reliably, so the wrong action is chosen with constant probability. This yields a total state-averaged error (i.e., regret) of order
$\Omega(|\theta|d_\theta^+)=\Omega(n^{-(m+1)/(2m)}).$

\subsubsection*{The $\set{\cM_\theta^-:\theta\in\Theta}$ Family}

For $\cM_\theta^-$, the reward function $r^-$ is defined by
$$r^-(x,a,(y,d)):= y\ \sgn(a)(|a|^q\wedge d^q) - ||a|-d|^p,\quad r^-(2,a,w) = 0$$
for $x\in[0,1]$ and $(a,(y,d))\in\A\times \W$, where $(x\wedge c) = \min{\set{x,c}}$. 

Then, define the similar law
\begin{equation}\label{eqn:psitheta-}
    \mrm{Law}_{\theta}^-(W_1) = \psi_\theta^-
: = \frac{1-\theta}{2}\delta_{(-1,d_\theta^-)} + \frac{1+\theta}{2}\delta_{(1,d_\theta^-)},\quad d_\theta^-:= |\theta|^{\frac{1}{p-q}}
\end{equation}
where the only difference compared to \eqref{eqn:psitheta+} is that the deterministic $D_1$ concentrates at $d_\theta^-$.

In this case, the optimal Q-function is given by \begin{equation}
Q_\theta^-(x,a)    =\int r^-(x,a,w)\psi_\theta(dw) = \theta\ \sgn(a)(|a|^q\wedge (d_\theta^-)^q) - ||a|-d_\theta^-|^p.\label{eqn:Qtheta-}
\end{equation}

We note that the intuition leading to the lower bound rate in this construction is as follows. The term $-||a|-d_\theta^-|^p$ creates two candidate maximizers at $a=\pm \, d_\theta^-$ and enforces the $p$-growth. The term
$\theta\,\sgn(a)\crbk{|a|^q\wedge (d_\theta^-)^q}$
then breaks the tie by favoring the action with $\sgn(a)=\sgn(\theta)$, while preserving $q$-H\"older continuity near $0$. Thus, when $\theta$ is small, the main difficulty is to determine which of $+d_\theta^-$ and $-d_\theta^-$, is optimal. The value loss from choosing the wrong branch is of order $|\theta|(d_\theta^-)^q=|\theta|^{p/(p-q)}.$
Therefore, if $|\theta|\asymp n^{-1/2}$, the sign of $\theta$ cannot be identified reliably, so the wrong branch is selected with constant probability. Since this ambiguity occurs uniformly over the state space, the resulting regret is of order $\Omega\bigl(|\theta|(d_\theta^-)^q\bigr)=\Omega\bigl(n^{-p/(2(p-q))}\bigr).$

\begin{remark}\label{rmk:static_case} Observe that when $x=2$, both rewards are $0$, and the transition satisfies $f(x,a,w)=2$ everywhere. Hence $X^\pi_k=2$ for all $k\geq 1$, so no further reward is collected after stage $0$. Thus, the policy optimization problem effectively reduces to a single-stage contextual stochastic optimization problem, where the context is the realized initial state. Therefore, together with the lower bounds in Theorem~\ref{thm:minimax_lb}, this shows that the optimal policy learning rates in Theorems~\ref{thm:continuous_action_unif_g_ub} and \ref{thm:continuous_action_ub} also apply to related stochastic optimization settings.
\end{remark} 

Next, we show that the instances defined above--specifically, their optimal Q-functions in \eqref{eqn:Qtheta+} and \eqref{eqn:Qtheta-}--satisfy Assumptions~\ref{assump:regularity}, \ref{assump:concentrability}, and \ref{assump:p_growth_margin_mass}.

\begin{proposition}\label{prop:instances_check_assumptions_134}
Fix any $\gamma\in(0,1)$, H\"older exponent $q\in(0,1]$, growth exponent $p>q$, and margin-mass exponent $m>0$. Then, any MDP instance within $\set{\cM_\theta^+,\cM_\theta^-:\theta\in\Theta}$  satisfies Assumptions \ref{assump:regularity}, \ref{assump:concentrability}, and \ref{assump:p_growth_margin_mass} with parameters $(m,p,q)$ and reference measure $\lambda_\mu:= \frac{1}{2}\mu + \frac{1}{2}\delta_2$. 
\end{proposition}

Before presenting the proof, we note that the verification process in the proof reflects the careful design of the hard instance family. In particular, these verifications complement the intuition underlying the constructions that lead to the lower bound rates (cf. the discussion after \eqref{eqn:Qtheta+} and \eqref{eqn:Qtheta-}).

\begin{proof}
Fix $\theta\in\Theta=[-1,1]$. We verify Assumptions
\ref{assump:regularity}, \ref{assump:concentrability}, and \ref{assump:p_growth_margin_mass}
for both instances $\cM_\theta^+$ and $\cM_\theta^-$.

\paragraph{Verifying Assumption \ref{assump:regularity}.} By the form of $(r^+,Q_\theta^+)$ and $(r^-,Q_\theta^-)$, the boundedness, measurability, and continuity requirements in Assumption~\ref{assump:regularity} are clearly satisfied. Moreover, by the one-stage nature of the dynamics and rewards noted in Remark~\ref{rmk:static_case}, a policy is optimal if and only if it maximizes the reward at time $0$. Hence, \eqref{eqn:Vstar_max_Qstar} holds in both cases.

\paragraph{Verifying Assumption \ref{assump:concentrability}.}
Fix any MDP instance $\cM_\theta^\dagger$ where $\dagger$ can be $+$ or $-$. Let $\pi:\X\to\A$ be any stationary policy
and let $\nu_\mu^\pi$ be the discounted occupancy measure as in Definition \ref{def:discounted_occupany_meas} with the
dependence on $\cM_\theta^\dagger$ suppressed.

As noted in Remark \ref{rmk:static_case}, under both families of MDPs, $X^\pi_0\sim\mu$ and $X^\pi_k=2$ for all $k\ge1$.
Hence, the state's marginal distribution at time $0$ is $\mu$, while for each $k\ge1$ it is $\delta_2$.
Therefore, for $\dagger = +$ and $-$,
\begin{align*}
\nu_\mu^\pi&
=(1-\gamma)\mrm{Law}_{\cM_\theta^\dagger}(X^\pi_0)+
(1-\gamma)\sum_{k=1}^\infty \gamma^k \mrm{Law}_{\cM_\theta^\dagger}(X^\pi_k)\\
&=
(1-\gamma)\mu + (1-\gamma)\sum_{k=1}^\infty \gamma^k \delta_2\\
&=
(1-\gamma)\mu+\gamma\delta_2,
\end{align*}
independent of $\pi$ and $\theta$.

Recall that $\lambda_\mu=\frac{1}{2}\mu+\frac{1}{2}\delta_2$ and $\gamma\in(0,1)$. Clearly $\nu_{\mu}^\pi\ll \lambda_\mu$ and
\[
\frac{d \nu_{\mu}^\pi}{d\lambda_\mu}(x)
=
\begin{cases}
2(1-\gamma) & \text{if }x\in[0,1],\\
2\gamma & \text{if }x=2,
\end{cases}
\]
is a version of the likelihood ratio. Consequently,
\[
\esssup_{\lambda_\mu}\frac{d \nu_{\mu}^\pi}{d\lambda_\mu}\le 2;
\]
i.e., Assumption~\ref{assump:concentrability} holds for both families with $C_{\mrm{conc}} = 2$.

\paragraph{Verifying Assumption \ref{assump:p_growth_margin_mass} for $\cM_\theta^+$.}
We work with $Q_\theta^+$ in \eqref{eqn:Qtheta+}. Before we verify Assumption \ref{assump:p_growth_margin_mass} for $\cM_\theta^+$, we establish some auxiliary facts. Note that for $x\in[0,1],\ a\in\A$,
\begin{equation}\label{eqn:Qtheta+_cases}
Q_\theta^+(x,a)=
\begin{cases}
x^m h_p(a), & x>d_\theta^+,\\
\theta+(x^m-\theta)h_p(a), & x\le d_\theta^+,
\end{cases}
\end{equation}
and $Q_\theta^+(2,a)=0$. Moreover, from \eqref{eqn:def_hp} we have for $a\in(-1,1)$,
\[
h_p(a)=\frac{1}{1+\crbk{\frac{1-a}{1+a}}^p},
\]
as well as $h_p(-1)=0$ and $h_p(1)=1$. Note that $a\to \frac{1-a}{1+a}$ is decreasing on $(-1,1]$, hence $h_p$ is
increasing on $[-1,1]$. With these preparations, we proceed to verify Assumption \ref{assump:p_growth_margin_mass}.

\smallskip
\noindent\emph{Optimal action sets and values.} We compute the optimal action set $\A_\theta^+(x)$ and optimal value function $V^+_\theta(x)$ for different cases of $x\in[0,1]$ and $\theta\in[-1,1]$.
\begin{itemize}
\item If $x>d_\theta^+$, then $Q_\theta^+(x,a)=x^m h_p(a)$ is maximized by $a=1$, so
$\A_\theta^+(x)=\{1\}$ and $V_\theta^+(x)=x^m$.
\item If $x\le d_\theta^+$ and $\theta>0$, then $x^m\le \theta/2$ so $x^m-\theta<0$, hence
$\theta+(x^m-\theta)h_p(a)$ is maximized by taking $h_p(a)=0$, i.e.\ $a=-1$.
Thus $\A_\theta^+(x)=\{-1\}$ and $V_\theta^+(x)=\theta$.
\item If $x\le d_\theta^+$ and $\theta<0$, then $x^m-\theta>0$, so the expression is maximized by $h_p(a)=1$, i.e.\ $a=1$.
Thus $\A_\theta^+(x)=\{1\}$ and $V_\theta^+(x)=x^m$.
\item If $\theta=0$, then $d_\theta^+=0$. For $x>0$, we are in the first case and $\A_0^+(x)=\{1\}$.
At $x=0$, $Q_0^+(0,a)=0$ for all $a$, hence $\A_0^+(0)=\A$.
\end{itemize}
Moreover, at $x=2$, $Q_\theta^+(2,\cd)\equiv 0$, so $\A_\theta^+(2)=\A$. 

Thus, the optimal action sets and the corresponding optimal values can be summarized as follows:
\begin{equation}\label{eqn:tu_A_V_breakdown}
\renewcommand{\arraystretch}{1.2}
\begin{array}{lll}
\A_\theta^+(x) \hspace{0.1in}& V_\theta^+(x)\hspace{0.1in} & \text{Condition} \\ \hline
\{1\}  & x^m    & x\in[0,1], \;\theta<0 \ \text{or}\ x>d_\theta^+\\
\{-1\} & \theta &x\in[0,1], \; \theta>0,\; x\le d_\theta^+\\
\A     & 0      & x=2 \ \text{or}\ (\theta,x)=(0,0)
\end{array}
\end{equation}

\smallskip
\noindent\emph{The $p$-growth condition.}
We claim that for all $x\in[0,1]$ and $a\in\A$,
\begin{equation}\label{eqn:pgrowth_plus_new}
V_\theta^+(x)-Q_\theta^+(x,a)\;\ge\;2^{-(p+1)}x^m\,\dist\!\big(a,\A_\theta^+(x)\big)^p.
\end{equation}
Also, $V_\theta^+(2)-Q_\theta^+(2,a)=0$ and $\dist(a,\A_\theta^+(2))=0$. Therefore, the $p$-growth condition \eqref{eqn:Q_err_pgrowth_in_action} holds with
\begin{equation}\label{eqn:gplus_new}
g_{p,\theta}^+(x):=
\begin{cases}
2^{-(p+1)}x^m,& x\in[0,1],\\
2,& x=2.
\end{cases}
\end{equation}

It remains to prove \eqref{eqn:pgrowth_plus_new}. To show this, we first note that for all $a\in[-1,1]$, $(1\pm a)^p\le 2^p$. Thus,
\begin{equation}\label{eqn:hp_basic_bounds}
\begin{aligned}
h_p(a)&=\frac{(1+a)^p}{(1+a)^p+(1-a)^p}\ge 2^{-(p+1)}(1+a)^p,\\
1-h_p(a)&=\frac{(1-a)^p}{(1+a)^p+(1-a)^p}\ge  2^{-(p+1)}(1-a)^p. 
\end{aligned}
\end{equation}

We now consider cases according to different optimal action sets in \eqref{eqn:tu_A_V_breakdown}.
\begin{itemize}
\item If $\A_\theta^+(x)=\{1\}$, then $V_\theta^+(x)=x^m$ and from \eqref{eqn:Qtheta+_cases}
\[
V_\theta^+(x)-Q_\theta^+(x,a)
=
\begin{cases}
x^m(1-h_p(a)), & x>d_\theta^+,\\
(x^m-\theta)(1-h_p(a)), & x\le d_\theta^+,\ \theta<0.
\end{cases}
\]
In both situations, applying the lower bound in \eqref{eqn:hp_basic_bounds} yields
\[
V_\theta^+(x)-Q_\theta^+(x,a)\ge x^m(1-h_p(a))
\ge2^{-(p+1)}x^m(1-a)^p.
\]
Since $\dist(a,\{1\})=1-a$ for $a\in[-1,1]$, this yields \eqref{eqn:pgrowth_plus_new}.

\item If $\A_\theta^+(x)=\{-1\}$, then we observe from \eqref{eqn:tu_A_V_breakdown} that necessarily $\theta>0$ and $x\le d_\theta^+$, so $x^m\le \theta/2$.
Moreover, in this case $V_\theta^+(x)=\theta$. Thus,
\begin{align*}  
V_\theta^+(x)-Q_\theta^+(x,a)
&=
\theta-\crbk{x^m h_p(a)+\theta(1-h_p(a))}\\
&=(\theta-x^m)h_p(a)\\
&\geq \frac{\theta}{2}h_p(a)\\
&\geq 2^{-(p+1)}x^m(a+1)^p.
\end{align*}
Since $\dist(a,\{-1\})=a+1$ for $a\in[-1,1]$, this also implies \eqref{eqn:pgrowth_plus_new}.

\item If $\A_\theta^+(x)=\A$, which this only occurs when $(\theta,x)=(0,0)$ or $x=2$, then
$\dist(a,\A_\theta^+(x))=0$ and $V_\theta^+(x)-Q_\theta^+(x,a)=0$, so \eqref{eqn:pgrowth_plus_new} holds trivially.
\end{itemize}
As these cases summarize all possible $\A_\theta^+(x)$, we have shown that \eqref{eqn:pgrowth_plus_new} holds for all $x\in[0,1]$ and $a\in\A$, checking the $p$-growth condition for $\cM_\theta^+$.

\smallskip
\noindent\emph{The margin-mass condition.}
With $g_{p,\theta}^+$ defined in \eqref{eqn:gplus_new}, fix $t\ge 0$.
If $t<2$, then $2\notin\{x\in\X:g_{p,\theta}^+(x)\le t\}$, and thus
\begin{align*}
\lambda_\mu\crbk{\{x\in\X:g_{p,\theta}^+(x)\le t\}}
&=
\frac{1}{2}\mu\crbk{\set{x\in[0,1]:2^{-(p+1)}x^m\le t}}\\
&=
\frac{1}{2}\mu\crbk{\set{x\in[0,1]:x\le (2^{p+1}t)^{1/m}}}\\
&\leq2^{\frac{p+1}{m}-1}t^{1/m}. 
\end{align*}   

If $t\ge 2$, then $g_{p,\theta}^+(2)=2\le t$. Also $g_{p,\theta}^+(x)\le 2^{-(p+1)}\le t$ for all $x\in[0,1]$,
so $\{x:g_{p,\theta}^+(x)\le t\}=\X$. Therefore, for $t\ge 2$,
\[
\lambda_\mu\crbk{\{x\in\X:g_{p,\theta}^+(x)\le t\}} = 1\leq 2^{-1/m}t^{1/m}. 
\]

Therefore, combining the two cases shows that \eqref{eqn:margin_mass} holds for all $t\ge 0$ for $\cM_\theta^+$ with exponent $m$
and constant $M=\max\set{2^{\frac{p+1}{m}-1},2^{-1/m}}$.

\paragraph{Verifying Assumption \ref{assump:p_growth_margin_mass} for $\cM_\theta^-$.}
Recall $Q_\theta^-$ computed in \eqref{eqn:Qtheta-}. 

\smallskip
\noindent\emph{Optimal action sets and values.}
Since $Q_\theta^-(x,\cd)$ is independent of $x\in[0,1]$, the maximizing action is the same for all $x\in[0,1]$. As noted after \eqref{eqn:Qtheta-}, it is not hard to see that
\begin{equation}\label{eqn:tu_A_V_breakdown_minus}
\renewcommand{\arraystretch}{1.2}
\begin{array}{lll}
\A_\theta^-(x) \hspace{0.1in}& V_\theta^-(x)\hspace{0.1in} & \text{Condition} \\ \hline
\{d_\theta^-\}  & |\theta|^{p/(p-q)}    & x\in[0,1], \;\theta\geq 0 \\
\{-d_\theta^-\} & |\theta|^{p/(p-q)} &x\in[0,1], \; \theta<0\\
\A     & 0      & x=2 
\end{array}
\end{equation}
and, when $\theta = 0$, $a_\theta^-$ is interpreted as 0. 

\smallskip
\noindent\emph{The $p$-growth condition.}
We define $$c_p:= \inf_{\rho \ge 0}\frac{1+|1-\rho|^p}{(1+\rho)^p} .$$
It is not hard to see that $c_p > 0$. We show that for all $\theta\in[-1,1]$ and $(x,a)\in\Z$,
\begin{equation}\label{eqn:ts_pgrowth_minus}
V_\theta^-(x)-Q_\theta^-(x,a)\ge c_{p}\dist\crbk{a,\A_\theta^-(x)}^p.
\end{equation} 
Hence the $p$-growth condition \ref{eqn:Q_err_pgrowth_in_action} hold with $g_p\equiv c_p$. 

Note that, by symmetry, we only need treat the case $\theta \geq 0$. The case $\theta< 0$ is equivalent when we change the variable $a$ to $-a$. Therefore, we will assume $\theta\geq 0$ and, in particular, $\A_\theta^{-}(x) = d_\theta^-$ when $x\in[0,1]$. 

We proceed to separately bound the growth for $a \ge 0$ and $a < 0$. When $a\ge 0$, by \eqref{eqn:tu_A_V_breakdown_minus} and \eqref{eqn:Qtheta-}, we have 
$$
\begin{aligned}V_\theta^-(x) - Q_\theta^-(x,a) &= |a - d_\theta^-|^p + \theta^{p/(p-q)} - \theta(a^q\wedge \theta^{q/(p-q)})\\
&=|a - d_\theta^-|^p + \theta(\theta^{q/(p-q)} - a^q\wedge \theta^{q/(p-q)})\\
&= 
\begin{cases}
    (d_\theta^- - a)^p + \theta((d_\theta^-)^q - a^q), &0\leq a<d_\theta^-,\\
    (a-d_\theta^-)^p, & d_\theta^- \leq a\le 1,
\end{cases} 
\end{aligned}$$
for all $x\in[0,1]$. Thus, when $a\geq 0$
\begin{equation}\label{eqn:tu_v_Q_minus_growth_age0}
V_\theta^-(x) - Q_\theta^-(x,a) \geq |a-d_\theta^-|^p = \dist(a,\A_\theta^-(x))^p \geq c_p\dist(a,\A_\theta^-(x))^p  
\end{equation}
for all $x\in[0,1]$ and $a\ge 0$, where we used $c_p\leq 1$. 

On the other hand, if $a<0$, 
$$V_\theta^-(x) - Q_\theta^-(x,a) \geq |d_\theta^--|a| |^p + (d_\theta^-)^p 
$$
Note that $\A_\theta^- = \set{d_\theta^-}$ which is non-negative. So for $a < 0$ and $x\in[0,1]$, $\dist(a,\A_\theta^-(x))= |a| +  d_\theta^-$. Using this fact, we bound
\begin{equation}\label{eqn:tu_v_Q_minus_growth_ale0}
\begin{aligned}V_\theta^-(x) - Q_\theta^-(x,a) &\geq \frac{|d_\theta^- -|a| |^p + (d_\theta^-)^p}{(|a| +  d_\theta^-)^p}\dist(a,\A_\theta^-(x))^p\\
&= \frac{|1- \rho_\theta |^p + 1}{(\rho_\theta +  1)^p}\dist(a,\A_\theta^-(x))^p\\
&\geq c_p \dist(a,\A_\theta^-(x))^p
\end{aligned}
\end{equation}
for $x\in[0,1]$, where $\rho_\theta = |a|/d_\theta^-$. 

Finally, combining \eqref{eqn:tu_v_Q_minus_growth_age0}, \eqref{eqn:tu_v_Q_minus_growth_ale0}, and the trivial case $x = 2$, we conclude \eqref{eqn:ts_pgrowth_minus} for $\theta\ge 0$. For $\theta < 0$, as noted earlier,  \eqref{eqn:ts_pgrowth_minus} follows from the same argument under a change of variable $b = -a$.

\smallskip
\noindent\emph{The margin-mass condition.}
Note that for the $-$ case, $g_p\equiv c_p$, so the margin mass condition \eqref{eqn:margin_mass} in Assumption \ref{assump:p_growth_margin_mass} holds with any $m> 0$ and
$M = c_p^{-1/m}$. 

\paragraph{Concluding the proof.}
To conclude, we have verified Assumptions \ref{assump:regularity},  \ref{assump:concentrability}, and \ref{assump:p_growth_margin_mass} for both instances $\cM_\theta^+$ 
and $\cM_\theta^-$ for any $\theta\in[-1,1]$, $m > 0$, $p>q>0$. This completes the proof.
\end{proof}

\subsection{Policy Learning Algorithms and Minimax Regret Lower Bounds}
\begin{definition}[Adapted policy learning and Q estimation algorithms]
    A data generating process is a filtered probability space $\set{\Omega,\cF,\set{\cD_n:n\ge 1},P}.$ 
    
    An adapted policy learning algorithm is a sequence of measurable maps $$\set{\hat \pi_n: \hat\pi_n\in m\set{ \cX\times \cD_n\ra\cA},\; n\ge 1}.$$ We denote the class of all adapted policy learning algorithm at time $n$ by $\hat \pi_n\in \mrm{PLA}_n$ where the dependence on the data generating process is suppressed.
    
    Similarly, an adapted Q-estimation algorithm is a sequence of measurable maps $$\set{\widehat  Q_n: \widehat  Q_n\in m\set{ \cZ\times \cD_n\ra\cB(\R)},\; n\ge 1}.$$
\end{definition}

Before we state the main minimax lower bound theorem, we redefine the policy regret notation (consistent with the definition in Definition \ref{def:regret}) to avoid potential ambiguity. For measurable function $\pi:\X\ra\A$ and $\cM_\theta^\dagger$, $\dagger = +,-$, let
\begin{equation}\label{eqn:def_reg_hard_M}
\reg\crbk{\pi;\cM_\theta^\dagger}:= V_\theta^\dagger(\mu) - V_\theta^{\pi,\dagger}(\mu) = V_\theta^\dagger(\mu) - \int_\X Q_\theta^\dagger(x,\pi(x))\mu(dx)
\end{equation}
where $\mu$ is the common uniform initial distribution, $Q_\theta^\dagger$ is defined in \eqref{eqn:Qtheta+} or \eqref{eqn:Qtheta+}, $V_\theta^{\pi,\dagger}(\mu)$ is defined as in \eqref{eqn:def_Vpi} using the instance $\cM_\theta^\dagger$, and the corresponding $V_\theta^\dagger$ is computed in the proof of Proposition \ref{prop:instances_check_assumptions_134}. Here, we note that the second equality follows from the one-stage structure as noted in Remark \ref{rmk:static_case}. 

Therefore, for $\hat\pi_n\in \mrm{PLA}_n$, $\hat\pi_n(\omega,\cd):\X\ra\A$ is measurable. Hence, the regret in \eqref{eqn:def_reg_hard_M} is well defined for $\hat\pi_n\in \mrm{PLA}_n$. With this notation in place, we now state the main lower bound theorem.

\begin{theorem}
\label{thm:minimax_lb}
Fix any $\gamma\in(0,1)$, H\"older exponent $q\in(0,1]$, growth exponent $p>q$, and margin-mass exponent $m>0$. For any $\theta\in\Theta = [-1,1]$ and $\dagger = +$ or $-$, there exists a data generating process $\set{\Omega,\cF,\set{\cD_n:n\ge1},P_\theta^\dagger}$  supporting an adapted Q-estimation algorithms $\set{\widehat Q_n^\dagger:n\geq 1}$ that satisfies Assumption \ref{assump:envelop_moment_bounds} with $Q^*$ replaced by $Q^\dagger_\theta$ and $P$ replaced by $P_\theta^\dagger$. At the same time, the uniform regret achievable by any adapted policy learning algorithms over each of the two families is lower bounded by 
\begin{equation}
\label{eq:lb_margin_mass}
\inf_{\hat\pi_n\in \mrm{PLA_n}}\max_{\theta\in\Theta}
E_\theta^+\sqbk{\reg\crbk{\hat\pi_n;\cM_\theta^+}}
\geq 
 Cn^{-(m+1)/(2m)}
\end{equation}
and
\begin{equation}
\label{eq:lb_growth}
\inf_{\hat\pi_n\in \mrm{PLA_n}}\max_{\theta\in\Theta}
E_\theta^-\sqbk{\reg\crbk{\hat\pi_n;\cM_\theta^-}}
\geq 
 Cn^{-p/(2(p-q))}
\end{equation}
where the constant $C$ is independent of $n$. 
\end{theorem}

In addition, by Proposition \ref{prop:instances_check_assumptions_134}, the instances $\cM_\theta^\dagger,\ \dagger = +,-$ further satisfy Assumptions \ref{assump:regularity}, \ref{assump:concentrability}, and \ref{assump:p_growth_margin_mass}. Therefore, for MDPs satisfying Assumptions~\ref{assump:regularity}--\ref{assump:p_growth_margin_mass}, the expected policy regret is lower bounded in a minimax sense by
$$
\Omega\crbk{\max\set{n^{-p/(2(p-q))},\,n^{-(m+1)/(2m)}}}
=
\Omega\crbk{n^{-\min\set{\frac{p}{2(p-q)},\,\frac{m+1}{2m}}}},
$$
which matches the upper bound in \eqref{eqn:continuous_action_ub_summary}, except in the case $p=q(m+1)$, where there is a $\log n$ gap.

We note that, once the hard MDP families are carefully constructed, the lower-bound proof of Theorem~\ref{thm:minimax_lb}, although technical, follows standard reduction arguments from the statistics literature; see, for example, \citet{tsybakov2008nonparametric}. At the same time, the proof also relies on a data-generating process motivated by a data-driven stochastic optimal control framework. This framework, which we study in detail in the next section, concerns learning an optimal control policy for a system whose dynamics are known, but whose noise distribution is unknown and must be estimated from data. 

In view of these considerations, we defer the proof of Theorem~\ref{thm:minimax_lb} to Appendix \ref{section:proof:thm:minimax_lb}.

\section{Sufficient Conditions for Assumptions \ref{assump:regularity} and \ref{assump:envelop_moment_bounds}}\label{section:suff_cond}

From the highlight in Remark~\ref{rmk:q_eq_0}, we see that the H\"older exponent $q$ in Assumption~\ref{assump:envelop_moment_bounds} is crucial for fast convergence of policy regret. In this section, we consider a class of policy learning environments that is widely applicable in operations settings, and show that Assumptions~\ref{assump:regularity} and \ref{assump:envelop_moment_bounds} hold with $q>0$ under mild regularity conditions on the model primitives. We then present two important OR problems that naturally satisfy these sufficient conditions. The core verification mechanisms extend naturally to other dynamical systems arising in operations applications. This suggests the broad applicability of our results and, more importantly, explains why fast convergence commonly arises in these settings.

Specifically, we consider the optimal policy learning problem for a Markov stochastic dynamic system driven by an i.i.d. noise sequence $\set{W_t:t\geq 1}$. The system evolution follows the state recursion
\begin{equation}
\label{eqn:noise_model_recursion}
X^\pi_{t+1}=f(X^\pi_t,A_t^\pi,W_{t+1}),
\end{equation}
where the measurable transition function $f:\X\times\A\times\W\to\X$ is known, while the distribution of the sequence $\set{W_t}$ must be inferred from data. This setting is common in data-driven control problems in OR, including applications with demand, workload arrival, and service-rate uncertainty.

Specifically, we consider an i.i.d.\ data set consisting of $2n$ samples $\{W_i^V,W_i^Q:i=1,\ds,n\}$ with law $\psi$ defined on an underlying probability space $(\Omega,\cF,P)$. To simplify notation, let $W\eqd W_1^V$ also be defined on $(\Omega,\cF,P)$, independent of $\{W_i^V,W_i^Q:i=1,\ds,n\}$. For all measurable $\phi$ and $U=V,Q$, define $\cF_n^U = \sigma(W_1^U,\ds W_n^U),$
$$
\psi_n^U:=\frac{1}{n}\sum_{i=1}^n\delta_{W_i^U}
\quad\text{and}\quad
E_n^U\phi(W)=\int_\W \phi(w)\psi_n^U(dw).
$$
To signify adaptivity, a random function with a superscript $U = V$ or $Q$ is generally $\cF_n^U$ measurable (e.g. $E[\phi(G^V_n(\theta,w))|\cF_n^V] = \phi(G^V_n(\theta,w))$ for all bounded measurable $\phi$).

We then construct the estimator in two steps. First, we solve the Bellman equation for the value function:
\begin{equation}\label{eqn:def_hatVn}
\widehat V_n(x)=\sup_{a\in\A}E_n^V[r(x,a,W)+\gamma \widehat V_n(f(x,a,W))]
\end{equation}
for all $x\in\X$. Then, using the remaining samples, we define
\begin{equation}\label{eqn:def_hatQn}
\widehat Q_n(z):=E_n^Q[r(z,W)+\gamma \widehat V_n(f(z,W))].
\end{equation}

We note that the sample split is used to make the verification transparent. The $V$-sample constructs the empirical Bellman fixed point $\widehat V_n$, and the independent $Q$-sample evaluates the corresponding action-value estimator. Conditional on $\cF_n^V$, $\widehat V_n$ is fixed when averaging over the $Q$-sample, which separates the regularity of $\widehat V_n$ from the empirical process error that determines $\delta_n$ and $\Lambda_n(q)$. Under broad conditions, this two-layer estimator satisfies Assumption~\ref{assump:envelop_moment_bounds}.

\subsection{H\"older Regularity}

We first identify regularity conditions under which the Bellman operator preserves H\"older continuity. This verifies Assumption~\ref{assump:regularity} and helps to establish the regularity of $\widehat V_n$.

For a function $g:\Ss\to\R$, where $\Ss$ is either $\X$ or $\Z$, and $\alpha\in(0,1]$, define the $\alpha$-H\"older coefficient by
\[
[g]_{\alpha,\Ss}:=
\sup_{u\neq u'\in\Ss}\frac{|g(u)-g(u')|}{|u-u'|^\alpha}.
\]

\begin{assumption}\label{assump:v_q_holder}
Assume the following conditions hold.
\begin{enumerate}[label=\textup{(\roman*)},leftmargin=*]
    \item For all $z\in\Z$, $\abs{E[r(z,W)]}\le r_\vee.$
    
    \item For some $\alpha_r\in(0,1]$, $r$ is weakly $\alpha_r$-H\"older continuous in $z$; i.e. there is $\ell_r <\infty$ such that
    \begin{equation}\label{eqn:r_weak_holder_z}
        E\abs{r(z,W)-r(z',W)}
        \le
        \ell_r |z-z'|^{\alpha_r}
    \end{equation}
    for all $z,z'\in\Z$.
    
    \item $f$ is weakly $\ell_f$-Lipschitz continuous in $z$; i.e.
    \begin{equation}\label{eqn:f_weak_lip_z}
        E\abs{f(z,W)-f(z',W)}
        \le
        \ell_f |z-z'|
    \end{equation}
    for all $z,z'\in\Z$.
\end{enumerate}
\end{assumption}
We note that these weak continuity and integrability assumptions are rather mild for OR applications, where many dynamic models naturally have a max-plus-linear structure; see Section~\ref{section:verification-or-models}.
Under these regularity conditions, we show in the following Theorem \ref{thm:holder_reg_of_V_Q} that the resulting optimal value and $Q$-functions are H\"older continuous. The proof follows from a successive approximation argument and is deferred to Appendix \ref{section:proof:thm:holder_reg_of_V_Q}. 
\begin{theorem}[H\"older continuity]
\label{thm:holder_reg_of_V_Q}
Suppose Assumption \ref{assump:v_q_holder} holds. Let $\alpha\in(0,1]$ satisfy $\alpha\le \alpha_r$ and $\gamma \ell_f^\alpha<1$, and define
\[
\ell_\alpha
:=
\frac{\ell_r^{\alpha/\alpha_r}(2r_\vee)^{1-\alpha/\alpha_r}}{1-\gamma \ell_f^\alpha}.
\]
Then $V^*$ and $Q^*$ as defined in \eqref{eqn:def_Qstar} are bounded and $\alpha$-H\"older continuous with $[V^*]_{\alpha,\X}\le \ell_\alpha$ and $[Q^*]_{\alpha,\Z}\le \ell_\alpha.$
Moreover, \eqref{eqn:Vstar_max_Qstar} holds; i.e.
$V^*(x)=\max_{a\in\A}Q^*(x,a)$ for all $x\in\X$.
\end{theorem}

Applying Theorem~\ref{thm:holder_reg_of_V_Q} to empirical measures supported on finitely many data points yields the following corollary.
\begin{corollary}
\label{cor:Q_Qn_regular}
Fix $n\ge 1$ and $\omega\in\Omega$. Suppose Assumption \ref{assump:v_q_holder} holds with the law of $W$ taken as $\psi$, $\psi_n^V(\omega,\cd)$, and $\psi_n^Q(\omega,\cd).$ Then the assertions in Assumption \ref{assump:regularity} hold.
\end{corollary}
\begin{proof}[Proof of Corollary \ref{cor:Q_Qn_regular}]
Applying Theorem \ref{thm:holder_reg_of_V_Q} with the population measure $\psi$ yields that $V^*\in C_b(\X)$, $Q^*\in C_b(\Z)$, and $V^*(x)=\max_{a\in\A}Q^*(x,a)$ for all $x\in\X$. 

Next, applying Theorem \ref{thm:holder_reg_of_V_Q} with $\psi$ replaced by $\psi_n^V(\omega,\cd)$ shows that $\widehat V_n(\omega,\cd)\in C_b(\X)$, since the Bellman equation in that theorem is precisely \eqref{eqn:def_hatVn}.

It remains to verify the corresponding statement in Assumption \ref{assump:regularity} for $\widehat Q_n(\omega,\cd)$. 

Let $z_k\to z$ in $\Z$. Since Assumption \ref{assump:v_q_holder} holds with reference measure $\psi_n^Q(\omega,\cd)$, part~(ii) implies that the map $z\to E_n^Q[r(z,W)]$ is continuous on $\Z$. Since $\widehat V_n(\omega,\cd)\in C_b(\X)$, bounded convergence then implies that
\[
E_n^Q[\widehat V_n(\omega,f(z_k,W))]\to E_n^Q[\widehat V_n(\omega,f(z,W))].
\]
Therefore, by \eqref{eqn:def_hatQn}, $\widehat Q_n(\omega,\cd)$ is continuous on $\Z$. 

Moreover, Assumption \ref{assump:v_q_holder} (i) with reference measure $\psi_n^Q(\omega,\cd)$ implies that $\sup_{z\in\Z}\abs{E_n^Q[r(z,W)]}<\infty$. This, together with the boundedness of $\widehat V_n(\omega,\cd)$, the continuity of $\widehat Q_n$, and the definition in \eqref{eqn:def_hatQn}, implies that $\widehat Q_n(\omega,\cd)\in C_b(\Z)$. 

This proves all the assertions in Assumption \ref{assump:regularity} regarding $V^*$, $Q^*$, and $\widehat Q_n(\omega,\cd)$.
\end{proof}

\subsection{Sufficient Conditions}

In this section, we establish sufficient conditions on the model primitives for Assumptions~\ref{assump:regularity} and \ref{assump:envelop_moment_bounds}. Leveraging the regularity properties in the previous section, we turn to stochastic error bounds that serve as envelopes for the errors in Assumption~\ref{assump:envelop_moment_bounds}. Before proving the main Theorem~\ref{thm:split_sample_est}, we explain how H\"older regularity yields envelopes satisfying the moments conditions in Assumption~\ref{assump:envelop_moment_bounds}. The key tool is the following empirical process bound for H\"older parametric classes. To streamline the discussion, we defer the proof to Appendix~\ref{section:proof:prop:empirical_holder_class}.

\begin{proposition}
\label{prop:empirical_holder_class}
Let $\Theta\subset\R^d$ be compact and let $g:\Theta\times\W\to\R$ be measurable. Assume that there exist measurable functions $R,L:\W\to\R_+$ and $\eta\in(0,1]$ such that
\[
|g(\theta,w)|\le R(w)
\quad\text{and}\quad
|g(\theta,w)-g(\theta',w)|\le L(w)|\theta-\theta'|^\eta
\]
for all $\theta,\theta'\in\Theta$ and $w\in\W$. In addition, assume that for every integer $k\ge2$, $E\sqbk{R(W)^k+L(W)^k}<\infty.$
Then, for every integer $k\ge2$, the empirical process error
$Z_n:=\sup_{\theta\in\Theta}\abs{(E_n-E)g(\theta,W)}$
satisfies
\[
E\sqbk{Z_n^k}\le \frac{C_{\mrm{emp},k}}{n^{k/2}}
\]
for some finite constant $C_{\mrm{emp},k}$ depending only on
$\crbk{\Theta,\eta,k,E[R(W)^k],E[L(W)^k]}$.
\end{proposition}

Therefore, if $V^*$ and $\widehat V_n$ are H\"older continuous, then, conditional on $\cF_n^V$, the estimation error $\|\widehat Q_n-Q^*\|$ can be viewed as the empirical process error over a H\"older parametric class indexed by $z\in\Z$. The action-wise regularity follows from a similar but more involved argument applied to the corresponding ratio process. Motivated by this intuition, the main sufficient conditions we use are summarized in Assumption~\ref{assump:suff_cond_A1A2}.

\begin{assumption}\label{assump:suff_cond_A1A2}
Assume that the following properties hold
\begin{enumerate}[label=\textup{(\roman*)},leftmargin=*]
    \item $\Z = \X\times \A\subseteq\R^{d_\X + d_\A}$ is compact. 
    \item There exists $\ell_{P} \geq 0$ and $\bar q\in[0,1]$, such that for all $x\in \X$ and $a,b\in \A$, $$d_{\mrm{TV}}(P(\cd|x,a),P(\cd|x,b))\le \ell_{P}|a-b|^{\bar q}.$$
    \item For some $z_0\in\Z$, $E[|r(z_0,W)|^k] < \infty$ for all integers $k\geq 2$. 
    \item There exists $\alpha_r\in(0,1]$ and a non-negative measurable function $L_r:\W\ra\R_+$ with $E[L_r(W)^k] < \infty$ for all integers $k\geq 2$, such that $|r(z,w) - r(z',w)| \leq L_r(w)|z-z'|^{\alpha_r}$ for all $z,z'\in\Z$.

    \item There exists $\ell_f\geq 0$ such that $|f(z,w) - f(z',w)| \leq \ell_f|z-z'|$ for all $z,z'\in\Z$.
\end{enumerate}
\end{assumption}

Under these conditions on the model primitives, the abstract requirements on $\widehat Q_n$ and $Q^*$ used in the regret analysis can be verified directly from $f$, $r$, and $\psi$. This is shown in the following theorem, whose proof is deferred to Appendix \ref{section:proof:thm:split_sample_est}. 

\begin{theorem}
\label{thm:split_sample_est}
Suppose Assumption \ref{assump:suff_cond_A1A2} is in force. 
Then, for every $\alpha\in(0,\alpha_r]$ such that $\gamma\ell_f^\alpha<1$, both $Q^*$ and $\widehat Q_n(\omega,\cd)$ are $\alpha$-H\"older continuous on $\Z$. Consequently, Assumption \ref{assump:regularity} holds. Moreover, define $\alpha^*:=\sup\{\alpha\in[0,\alpha_r]:\gamma\ell_f^\alpha<1\}$, then Assumption \ref{assump:envelop_moment_bounds} holds for every exponent $ q\in[0,\bar q]$ and $q < \alpha^*$.
\end{theorem}

 In particular, as $\alpha^* > 0$, Theorem \ref{thm:split_sample_est} implies that if (ii) holds with $\bar q > 0$, then the split-sample estimator satisfies Assumptions \ref{assump:envelop_moment_bounds} with some $ q>0$. This implies a fast convergence rate for the regret of the greedy policy.

\subsection{Total Variation Regularity in Operations Examples}
\label{section:verification-or-models}

In Assumption~\ref{assump:suff_cond_A1A2}, the smoothness and integrability conditions on the reward and dynamics are usually easy to verify. The nontrivial requirement is the total-variation regularity condition in the action, namely part (ii). This is also the key factor determining whether a fast convergence rate is possible from Theorem \ref{thm:split_sample_est}. We show that this condition holds broadly with $\bar q=1$ in OR settings through two continuous-state, continuous-action examples: a multi-item lost-sales inventory control model and a dynamic service allocation problem.

We impose the following regularity condition on the exogenous noise, which may represent demand in the inventory problem or arriving workload in the service allocation problem. Assume that the noise distribution admits a density $p$ on $\R^{d_\W}$ such that for some $\bar q\in(0,1]$ and $\ell<\infty$,
\begin{equation}\label{eqn:cond_on_density}
\norm{p(\cdot+h)-p}_{L^1(\R^{d_\W})}
\le
\ell |h|^{\bar q}
\end{equation}
for all $h\in\R^{d_\W}$. We note that this assumption permits arbitrary dependence across demand coordinates. In particular, it holds with $\bar q=1$ when $p\in BV(\R^{d_\W})$ has bounded variation; i.e., there exists a vector of finite measures $\nu = (\nu_1,\ds,\nu_m)$ such that $\int(\nabla\cd\phi) d p_D = -\int \phi^\top d\nu_D$ for all $\phi\in C_c^1$ that is continuously differentiable with compact support. 

Indeed, we consider a $C_c^\infty$ mollifier $\rho_\epsilon$ that integrates to 1, supported on the ball with radius $\epsilon$ centered at 0 (e.g., the bump function). Let $p_\epsilon := p*\rho_\epsilon$ be the smoothed density. Then, differentiate w.r.t. $x_i$, $$\del_{i}p_\epsilon(x) = \int \del_i \rho_\epsilon(x-y)p(dy) = \int \rho_\epsilon(x-y)\nu_i(dy) = \nu_i*\rho_\epsilon$$
By the fundamental theorem of calculus, the smoothed density satisfies the estimate
$$\begin{aligned}\norm{p_\epsilon(\cd+h) - p_\epsilon}_{L^1(\R^{d_\W})} &= \int \abs{\int_0^1 h^\top\nabla p_\epsilon(x +th)}  dx \\
&\leq |h| \abs{\int \nu*\rho_\epsilon(x) dx}\\
&= |h|\abs{\iint \rho_\epsilon(x-y)dx\nu(dy)}\\
&=|h||\nu|(\R^{d_\W})
\end{aligned}$$
where $|\nu|$ denotes the total variation measure, and the last equality follows from $\rho_\epsilon(\cd-y)$ integrates to 1. Since $p\in L^1(\R^{d_\W})$, taking $\epsilon\da0$ yields \eqref{eqn:cond_on_density}.

We note that the Sobolev space $W^{1,1}(\R^{d_\W})\subset BV(\R^{d_\W})$ contains a rich family of densities. Moreover, regular densities with jump discontinuities (e.g., uniforms on regular geometries such as cubes or spheres) usually also have bounded variation. 

With this condition in mind, we look at the following examples. 

\begin{example}\label{example:inventory}
Consider a Markovian multi-item lost-sales inventory control problem with continuous inventory levels.  The state space is
$\X=\set{(x_1,\ds,x_k)\in\R_+^k:\ 0\le x_i\le b_i,\ i=1,\ds,k}$
where $x_i$ and $b_i$ denote the inventory level and storage capacity of item $i$, respectively.  A natural action space is
$\A =\set{(a_1,\ds,a_k)\in\R_+^k:\ 0\le a_i\le b_i,\ i=1,\ds,k}$,
where $a_i$ is the amount of item $i$ ordered and delivered at the beginning of the period.  We assume no lead time between ordering and the arrival of replenishment inventory. Let $X_t$ be the inventory level at the beginning of day $t$, before the arrival of the replenishment order $A_t$. Then, naturally, the inventory levels evolve as
$$
X_{t+1}=\crbk{\min\set{X_t+A_t,b}-D_{t+1}}_+,
$$
where the minimum and positive part are applied componentwise, and $D_{t+1}$ denotes demand. The demand process $\set{D_t\in\R_+^k:t\geq 1}$ is i.i.d., with possibly correlated components. Then the transition kernel is given by $
P(\cdot| x,a)
=\text{Law}\sqbk{(\min\set{x+a,b}-D)_+}.
$

Now, assume that the distribution of $D_t$ satisfies \eqref{eqn:cond_on_density}. Since total variation is non-increasing under the pushforward mapping $(\cd)_+$,
\[
\begin{aligned}
d_{\TV}\crbk{
P(\cdot| x,a),
P(\cdot| x,a')
} &\le
d_{\TV}\crbk{
\text{Law}\crbk{\min\set{x+a,b}-D},
\text{Law}\crbk{\min\set{x+a',b}-D}
} \\
&=
\frac12
\int_{\R^m}
\abs{
p\crbk{\min\set{x+a,b}-u}
-
p\crbk{\min\set{x+a',b}-u}
}
\,du \\
&\le
\frac{\ell}{2}|\min\set{x+a,b}-\min\set{x+a',b}|^{\bar q}\\
&\leq \frac{\ell}{2}|a-a'|^{\bar q},
\end{aligned}
\]
where the last inequality follows from the fact that the $\min$ is $1$-Lipschitz. Thus, Assumption~6 part (ii) holds with $\ell_P=\ell/2$ and exponent $\bar q$.
\end{example}

\begin{example}\label{example:workload}
Consider a workload formulation of a dynamic service allocation problem. We use the same state spaces
$\X=\set{(x_1,\ds,x_k)\in\R_+^k:\ 0\leq x_i\leq b_i,\ i=1,\ds,k}$, where $x_i$ and $b_i$ denote the remaining workload, and buffer capacity of class $i$, respectively. We consider an action space representing possible service-allocation vectors,
$$
\A=\set{(a_1,\ds,a_k)\in\R_+^k:\ \sum_{i=1}^k a_i\leq \zeta,\ 0\leq a_i\leq b_i,\ i=1,\ds,k},
$$
where $a_i$ denotes the service capacity allocated to class $i$, and $\zeta$ is the total available service capacity. 

Let $X_t$ denote the remaining workload vector at the beginning of period $t$.  During the period, a new workload vector $\Xi_{t+1}$ arrives, and the controller allocates service capacity $A_t$ across the $k$ classes.  The allocated service reduces the combined old and newly arrived workload, but the workload cannot become negative and is truncated by the finite buffer capacity.  Thus, the workload evolves as
\[
X_{t+1}=\min\set{(X_t+\Xi_{t+1}-A_t)_+,b}.\]
The arrival process $\set{\Xi_t\in\R_+^m:t\ge 1}$ is assumed to be i.i.d., and the components of $\Xi_t$ may be correlated. The transition kernel is thus given by
$
P(\cdot| x,a)
=
\text{Law}(\min\set{(x+\Xi-a)_+,b}).
$

Now, assume that the distribution of $\Xi_t$ satisfies \eqref{eqn:cond_on_density}. By the same argument as in Example \ref{example:inventory},
\[
\begin{aligned}
d_{\TV}\crbk{
P(\cdot| x,a),
P(\cdot| x,a')
}
&\le
d_{\TV}\crbk{
\text{Law}\crbk{x+\Xi-a},
\text{Law}\crbk{x+\Xi-a'}
} \\
&=
\frac12
\int_{\R^m}
\abs{
p\crbk{u-x+a}
-
p\crbk{u-x+a'}
}
\,du \\
&\le
\frac{\ell}{2}|a-a'|^{\bar q}.
\end{aligned}
\]
Thus, Assumption~6 part (ii) holds with $\ell_P=\ell/2$ and exponent $\bar q$.
\end{example}

\begin{remark}
Both examples satisfy Assumption~6(ii) through the same mechanism. In both cases, for fixed $x$, changing the action translates the exogenous noise distribution before applying a fixed measurable map. Thus, the desired action regularity follows from the $L^1$ translation regularity of the density of the exogenous noise, as in \eqref{eqn:cond_on_density}. This phenomenon extends beyond these two settings to many max-plus-linear systems \citep{cohen1985linear}, which commonly arise in OR applications. However, if the exogenous noise has discrete masses, then an arbitrarily small translation can induce a TV distance proportional to the mass, and the TV regularity may fail.
\end{remark}

\subsection*{Acknowledgments}
J. Blanchet gratefully acknowledges support from DoD through the grant ONR 1398311, also support from NSF via grants 2229012, 2312204, and 2403007 is gratefully acknowledged.
\bibliographystyle{apalike}
\bibliography{ref}

\appendixpage
\appendix

\section{A Generalized H\"older Inequality}
\begin{lemma}
\label{lemma:generalized_holder}
Let $X,Y$ be random variables. Let $p,q\in[1,\infty]$ satisfy $\frac1p+\frac1q\leq 1.$
Then
\[
\E\sqbk{\abs{XY}}
\le
\E\sqbk{\abs{X}^p}^{1/p} 
\E\sqbk{\abs{Y}^q}^{1/q},
\]
with the usual convention that when $p=\infty$ or $q=\infty$, the corresponding factor is interpreted as the essential supremum.
\end{lemma}

\begin{proof}
Let $$r = \frac{1}{1-\frac{1}{p}-\frac{1}{q}}\ge 1$$ and $r = \infty$ if $\frac1p+\frac1q = 1.$ Apply the three-factor H\"older inequality to the random variables $\abs{X}$, $\abs{Y}$, and $1$ with exponents $p,q,r$. Since
\[
\frac1p+\frac1q+\frac1r=1,
\]
we obtain
\[
\E\abs{XY}
=
\E\sqbk{\abs{X}\abs{Y}\cdot 1}
\le
\E\sqbk{\abs{X}^p}^{1/p} 
\E\sqbk{\abs{Y}^q}^{1/q} 
\E\sqbk{1^r}^{1/r}.
\]
Hence, 
\[
\E{\abs{XY}}
\le
\E\sqbk{\abs{X}^p}^{1/p} 
\E\sqbk{\abs{Y}^q}^{1/q}.
\]
\end{proof}

\section{Proofs of Auxiliary Lemmas for Theorem \ref{thm:continuous_action_ub}}
\label{section:aux_thm_caub}
This section collects proofs of auxiliary lemmas used in the proof of Theorem \ref{thm:continuous_action_ub}.

\subsection{Proof of Lemma \ref{lemma:moment_truncation_bound}}\label{section:proof:lemma:moment_truncation_bound}
\begin{proof}
Fix $\alpha\ge 1/m$ and $u>0$, $v>0$. Define the threshold
\[
s := \crbk{\frac{v}{2u}}^{1/\alpha}
\qquad\text{and hence}\qquad v s^{-\alpha}=2u.
\]
Write
\begin{equation}\label{eqn:tu_min_sep_bd}\E\sqbk{\min\set{2u,\ vG^{-\alpha}}}
=
2u \mathbb P(G\le s) + v\E\sqbk{G^{-\alpha}\1\{G>s\}},
\end{equation}
where $G^{-\alpha}=+\infty$ on $\{G=0\}$.

If $s\ge 1$, then $v\ge 2u$ and 
$G^{-\alpha}\1\{G>s\} \leq \1\{G>s\}$. So, \eqref{eqn:tu_min_sep_bd} implies that $\E\sqbk{\min\set{2u,\ vG^{-\alpha}}}\leq v$ when $s\ge 1$. Since both stated upper bounds in Lemma \ref{lemma:moment_truncation_bound} dominate $v$, it suffices to treat the case $s\in(0,1)$.

Assume now $s\in(0,1)$. By assumption, the first term satisfies
\begin{equation}\label{eqn:tu_min_sep_uPG_bd}
2u \mathbb P(G\le s)\le 2Mu s^{1/m}.
\end{equation}

For the second term, use the tail integral identity:
\begin{equation}\label{eqn:tu_EG_as_int_bd}
\begin{aligned}
\E\sqbk{G^{-\alpha}\1\{G>s\}}
&= \int_0^\infty \mathbb P\crbk{G^{-\alpha}\1\set{G>s}> t} dt\\
&=\int_0^\infty \mathbb P\crbk{s<G < t^{-1/\alpha}} dt\\
&\leq 1 + \int_1^ {s^{-\alpha}} \mathbb P\crbk{s<G < t^{-1/\alpha}} dt
\end{aligned}
\end{equation}
where the contribution to the integral from $t\in(0,1)$ is at most 1, and when $t \geq s^{-\alpha}$, $t^{-1/\alpha}\leq s$, leading to a 0 integrand. For $t\in(1,s^{-\alpha})$, we have $t^{-1/\alpha}\in(s,1)$. Hence, by assumption,
\begin{equation}\label{eqn:tu_prob_int_bd}
\int_1^{s^{-\alpha}} \mathbb P\crbk{s<G<t^{-1/\alpha}} dt
\le
M\int_1^{s^{-\alpha}} t^{-1/(m\alpha)} dt.
\end{equation}
Using this bound, we now conclude separately for $\alpha>1/m$ and $\alpha = 1/m$. 

\paragraph{Case $\alpha>1/m$.} In this case, $1/(m\alpha)\in(0,1)$ and \eqref{eqn:tu_prob_int_bd} is bounded by
\[
M\int_1^{s^{-\alpha}} t^{-1/(m\alpha)} dt
= \frac{Mm\alpha}{m\alpha-1}\crbk{s^{1/m-\alpha}-1} \leq C s^{1/m-\alpha},
\]
where $C<\infty$ depending only on $(M,m,\alpha)$. Therefore, from \eqref{eqn:tu_EG_as_int_bd}, 
\[
v\E[G^{-\alpha}\1\{G>s\}]
\le v + C v s^{1/m-\alpha}
= v + C u^{1-1/(m\alpha)}v^{1/(m\alpha)}. 
\]

Moreover, from \eqref{eqn:tu_min_sep_uPG_bd}, 
\[
2u \mathbb P(G\le s)\le 2Mu s^{1/m}
= 2M 2^{-1/(m\alpha)}u^{1-1/(m\alpha)}v^{1/(m\alpha)}
\]
matching the power of the expectation term. 

Therefore, these bounds and \eqref{eqn:tu_min_sep_bd} together implies Lemma \ref{lemma:moment_truncation_bound} (i).

\paragraph{Case $\alpha=1/m$.} In this case, the integral in \eqref{eqn:tu_prob_int_bd} is
\[
M\int_1^{s^{-\alpha}} t^{-1/(m\alpha)} dt
= M\int_1^{s^{-m}} t^{-1} dt
= Mm\log(s^{-1}).
\]
Since $s=(v/(2u))^m < 1$, we have $\log(s^{-1})=\log(2u/v)$, and therefore
\[
v\E[G^{-1/m}\1\{G>s\}]
\le v + Mm v\log\crbk{\frac{2u}{v}}.
\]
Moreover, noting that when $s < 1$
\[
\log\crbk{\frac{2u}{v}}= \max\set{0, \log\crbk{\frac{2u}{v}}},
\]

Now, \eqref{eqn:tu_min_sep_uPG_bd} implies that $2u \mathbb P(G\le s)\le 2Mu s^{1/m}=Mv$. 
These bounds together gives Lemma \ref{lemma:moment_truncation_bound} (ii).
\end{proof}

\subsection{Proof of Lemma \ref{lemma:critical_log}}\label{section:proof:lemma:critical_log}
\begin{proof}
To simplify notation, write
\[
\log_+(x):=\max\set{0, \log x}
\]
for all $x > 0$. For $n\geq 2$, let
\[
\theta_n:=\frac{1}{1+\log n}\in(0,1). 
\]
For any $x>0$ and any $\theta>0$, we have
\[
\log_+(x)\le \frac{x^\theta}{\theta}.
\]
Indeed, this is trivial when $x\le 1$. For $x>1$, $d_x\log(x) = x^{-1} \leq d_x(x^\theta/\theta) = x^{\theta-1}$ and $\log(1)= 0 < 1/\theta$. So, $\log x\le x^\theta/\theta$ when $x>1$.

Applying this with $x=2\delta_n/\Lambda_n(q)^\beta > 0$ and $\theta=\theta_n$, we get
\[
\Lambda_n(q)^\beta \log_+\crbk{\frac{2\delta_n}{\Lambda_n(q)^\beta}}
\le
\frac{1}{\theta_n}\Lambda_n(q)^\beta \crbk{\frac{2\delta_n}{\Lambda_n(q)^\beta}}^{\theta_n}
=
\frac{2^{\theta_n}}{\theta_n}\delta_n^{\theta_n}\Lambda_n(q)^{\beta(1-\theta_n)}.
\]
Therefore
\begin{equation}
\label{eqn:critical_log_reduce_to_mixed_moment}
\E\sqbk{\Lambda_n(q)^\beta \log_+\crbk{\frac{2\delta_n}{\Lambda_n(q)^\beta}}}
\le
\frac{2^{\theta_n}}{\theta_n}
\E\sqbk{\delta_n^{\theta_n}\Lambda_n(q)^{\beta(1-\theta_n)}}.
\end{equation}

Now set
\[
p_n:=\frac{2}{\theta_n},
\qquad
q_n:=\frac{k_\beta}{\beta(1-\theta_n)}.
\]
Since $k_\beta=\ceil{\beta+1}>\beta$ and $\theta_n\in(0,1)$,
\[
\frac{1}{p_n}+\frac{1}{q_n}
=
\frac{\theta_n}{2}+\frac{\beta(1-\theta_n)}{k_\beta}
=
\frac{\beta}{k_\beta}
+
\theta_n\crbk{\frac12-\frac{\beta}{k_\beta}}
<1.
\]
Hence, the generalized H\"older inequality in Lemma \ref{lemma:generalized_holder} gives
\[
\E\sqbk{\delta_n^{\theta_n}\Lambda_n(q)^{\beta(1-\theta_n)}}
= \E\sqbk{\delta_n^{2/p_n}\Lambda_n(q)^{k_\beta/q_n}} \le
\E\sqbk{\delta_n^2}^{1/p_n}
\E\sqbk{\Lambda_n(q)^{k_\beta}}^{1/q_n}.
\]
Therefore, by Assumption \ref{assump:envelop_moment_bounds},
\begin{equation} \label{eqn:tu_delta_Lambda_moment_bd}
\E\sqbk{\delta_n^{\theta_n}\Lambda_n(q)^{\beta(1-\theta_n)}}
\le
C n^{-\theta_n/2}n^{-\beta(1-\theta_n)/2}
=
C n^{-\beta/2}n^{(\beta-1)\theta_n/2},
\end{equation}
where $C<\infty$ depends only on $(\beta,C_\delta,C_{\Lambda,k_\beta})$.

Since
\[
\theta_n\log n=\frac{\log n}{1+\log n}\le 1,
\]
we have
\[
n^{(\beta-1)\theta_n/2}
=
\exp\crbk{\frac{\beta-1}{2}\theta_n\log n}
\le
\exp\crbk{\frac{\beta-1}{2}}.
\]
Also, $2^{\theta_n}\le 2$ and
\[
\frac{1}{\theta_n}=1+\log n \leq 2\log 2+\log n\le 3\log n,
\qquad n\ge 2.
\]
Substituting these bounds into \eqref{eqn:tu_delta_Lambda_moment_bd} and \eqref{eqn:critical_log_reduce_to_mixed_moment}, we obtain
\begin{align*}
\E\sqbk{\Lambda_n(q)^\beta\log_+\crbk{\frac{2\delta_n}{\Lambda_n(q)^\beta}}}
&\le
C \frac{2^{\theta_n}}{\theta_n} n^{-\beta/2}n^{(\beta-1)\theta_n/2}\\
&\leq 6e^{(\beta-1)/2}C  n^{-\beta/2}\log n 
\end{align*}
for all $n\ge 2$. This proves Lemma \ref{lemma:critical_log}
\end{proof}

\subsection{Proof of Lemma \ref{lemma:inv_moment}}\label{section:proof:lemma:inv_moment}
\begin{proof}
When $\alpha = 0$, Lemma \ref{lemma:inv_moment} holds trivially. For $\alpha\in(0,1/m)$, using the tail integral identity as in \eqref{eqn:tu_EG_as_int_bd} and the assumption on $G$, we get
\begin{align*}
\E\sqbk{G^{-\alpha}\1\{G>0\}}
&\le
1 + \int_1^\infty \mathbb P(0<G<t^{-1/\alpha}) dt.\\
&\leq 1+M\int_1^\infty t^{-1/(m\alpha)} dt
\\
&=1+\frac{Mm\alpha}{1-m\alpha}.
\end{align*}
\end{proof}
\section{Proof of Theorem \ref{thm:minimax_lb}}\label{section:proof:thm:minimax_lb}

\begin{proof}
The proof is structured as follows. First, we show that for each fixed $\theta\in\Theta$, $m > 0$, $p>q\ge 0$, and $\dagger\in\set{+,-}$, there exist adapted
Q-estimators $\set{\widehat Q_n:n\ge1}$ satisfying Assumption \ref{assump:envelop_moment_bounds}
with $Q^*$ replaced by $Q_\theta^\dagger$. Second, we establish the lower bounds
\eqref{eq:lb_margin_mass} and \eqref{eq:lb_growth}.

\subsubsection*{Data generating processes and adapted Q-estimators.}
Fix $\theta\in\Theta=[-1,1]$ and $\dagger\in\set{+,-}$. We work on the canonical space $(\Omega=\W^{\N},\cF=\cB(\W)^{\N},P_\theta^\dagger:=(\psi_\theta^\dagger)^{ \N})$, where $\psi_\theta^+$ and $\psi_\theta^-$ are defined in \eqref{eqn:psitheta+}
and \eqref{eqn:psitheta-}, respectively, and $\N$ is the set of positive integers. 

Define the sequence $\set{W_i=(Y_i,D_i):i=1,2,\ds}$ by the coordinate projection mapping $W_i(\omega) = \omega_i = (y_i,d_i)\in \W$. Let $\cD_n = \sigma\set{W_i:i=1,\ds,n}$. Thus, under $P_\theta^\dagger$, $\set{W_i:i\in\N}$ are i.i.d. with common law $\psi_\theta^\dagger$. 

Define the empirical means
\[
\bar Y_n:=\frac{1}{n}\sum_{i=1}^n Y_i,\quad\text{and}\quad
\bar D_n:=\frac{1}{n}\sum_{i=1}^n D_i.
\]
By construction, $\bar Y_n$ and $\bar D_n$ are $\cD_n$-measurable, and $\bar D_n=d_\theta^\dagger$ a.s.$P_\theta^\dagger$ for all $n\geq 1$, where $d_\theta^\dagger$ is defined in \eqref{eqn:psitheta+} and \eqref{eqn:psitheta-}.

We now define the adapted plug-in estimators. For the $+$ family, let
\begin{equation}\label{eqn:minimax_lb_Qhat_plus}
\widehat Q_n^+(x,a)
:=
\begin{cases}
x^m h_p(a)+\bar Y_n\1\set{x\le \bar D_n}\crbk{1-h_p(a)}, & x\in[0,1],\\
0, & x=2,
\end{cases}
\end{equation}
where $h_p$ is defined in \eqref{eqn:def_hp}. For the $-$ family, let
\begin{equation}\label{eqn:minimax_lb_Qhat_minus}
\widehat Q_n^-(x,a)
:=
\begin{cases}
\bar Y_n\,\sgn(a)\crbk{|a|^q\wedge \bar D_n^q}-\big||a|-\bar D_n\big|^p, & x\in[0,1],\\
0, & x=2.
\end{cases}
\end{equation}
Both $\widehat Q_n^+$ and $\widehat Q_n^-$ are $\cD_n$-adapted Q-estimation algorithms in the sense of
the preceding definition. Also, for each fixed $\omega$ and $x$, the maps
$a\to \widehat Q_n^\dagger(\omega,x,a)$ are continuous on $\A$.

We note that since $Y_i\in\set{-1,1}$ and $E_\theta[Y_i]=\theta$ under either family, $\bar Y_n-\theta$ is centered and
sub-Gaussian with variance proxy $1/n$. Therefore, by standard sub-Gaussian moment bounds
(see, e.g., \citet[Section 2.1]{wainwright2019HD}), for every integer $k\ge 2$, there exists a constant
$C_k<\infty$, depending only on $k$, such that
\begin{equation}\label{eqn:tu_subgaussian_moments}
E_\theta\sqbk{|\bar Y_n-\theta|^k}\le C_k n^{-k/2},
\qquad k\ge 2.
\end{equation}

We proceed with the proof by considering the $+$ and $-$ families separately.

\subsubsection*{Proof of Theorem \ref{thm:minimax_lb} for the $\set{\cM_\theta^+:\theta\in\Theta}$ family.}

\paragraph{Verifying Assumption \ref{assump:envelop_moment_bounds}.}
Since $\bar D_n=d_\theta^+$ a.s.$P_\theta^+$, \eqref{eqn:minimax_lb_Qhat_plus} and \eqref{eqn:Qtheta+} give
\[
\widehat Q_n^+(x,a)-Q_\theta^+(x,a)
=
(\bar Y_n-\theta)\1\set{x\le d_\theta^+}\crbk{1-h_p(a)},
\]
for all $(x,a)\in \Z$ a.s.$P_\theta^+$. Since $h_p(a)\in[0,1]$ for all $a\in\A$,
\begin{equation}\label{eqn:tu_ub_Qtheta+_delta}
\sup_{z\in\Z}\abs{\widehat Q_n^+(z)-Q_\theta^+(z)}
\le
\abs{\bar Y_n-\theta},
\quad \text{a.s.$P_\theta^+$.}
\end{equation}

Next, fix $x\in\X$ and $a\neq b$. Then
\begin{align*}
&\frac{\abs{\crbk{\widehat Q_n^+(x,a)-Q_\theta^+(x,a)}-\crbk{\widehat Q_n^+(x,b)-Q_\theta^+(x,b)}}}{|a-b|^q}=
\abs{\bar Y_n-\theta}\1\set{x\le d_\theta^+}
\frac{\abs{h_p(a)-h_p(b)}}{|a-b|^q}.
\end{align*}
If $q=0$, then the last quotient is bounded by $2$. If $q\in(0,1]$, then
\[
h_p(a)=\frac{(1+a)^p}{(1+a)^p+(1-a)^p},
\]
the denominator is bounded away from $0$ on $[-1,1]$. Recall that $p>q$, the maps
$a\ra (1+a)^p$ and $a\ra (1-a)^p$ are $q$-H\"older on $[-1,1]$.
Hence there exists $0 <\ell_{p,q}^+ < \infty$, depending only on $(p,q)$, such that for all $a,b\in[-1,1]$
\[
|h_p(a)-h_p(b)|\le \ell_{p,q}^+|a-b|^q. 
\]
Consequently,
\begin{equation}\label{eqn:tu_ub_Qtheta+_Lambda}
\sup_{x\in\X}\sup_{a\neq b}
\frac{\abs{\crbk{\widehat Q_n^+(x,a)-Q_\theta^+(x,a)}-\crbk{\widehat Q_n^+(x,b)-Q_\theta^+(x,b)}}}{|a-b|^q}
\le
(\ell_{p,q}^+\vee 2)\abs{\bar Y_n-\theta},
\qquad \text{a.s.$P_\theta^+$.}
\end{equation}

Thus,  in view of \eqref{eqn:tu_ub_Qtheta+_delta} and \eqref{eqn:tu_ub_Qtheta+_Lambda}, we define
\[
\delta_n^+:=\abs{\bar Y_n-\theta}+n^{-17}
\quad\text{and}\quad
\Lambda_n^+(q):=(\ell_{p,q}^+\vee 2)\abs{\bar Y_n-\theta}+n^{-17}.
\]
Applying \eqref{eqn:tu_subgaussian_moments}, we conclude that Assumption \ref{assump:envelop_moment_bounds} holds for the $+$ family with positive measurable envelopes $\delta_n^+$ and $\Lambda_n^+(q)$.

\paragraph{Minimax lower bound for $\set{\cM_\theta^+:\theta\in\Theta}$.}
Fix $n\ge 1$ and let $\hat\pi_n\in \mrm{PLA}_n$ be an arbitrary adapted policy learning algorithm.
Since $\hat\pi_n$ is $\cD_n$-adapted and $\cD_n=\sigma\set{W_1,\ds,W_n}$, by the Doob–Dynkin lemma we may identify $\hat\pi_n$ with a measurable map in $m\set{\cW^n\times\cX\ra \cA }$, and we suppress the dependence on $w_1,\ds,w_n$ in the notation below.

To simplify notation, let
\[
\Psi_{\theta}^n(d\bd{w}_n) :=\Psi_{\theta}^n(dw_1,\ds,dw_n):=\psi_\theta^+(dw_1)\psi_\theta^+(dw_2)\ds \psi_\theta^+(dw_n);
\]
i.e. $\Psi^n_\theta$ is the law of $W_1,\ds,W_n$ under $P_\theta^+$.

We consider $\theta = \pm\,\xi$ where
\[
\xi:=\frac{1}{4\sqrt n},\quad \text{and recall that}\quad d_\xi^+=\crbk{\frac{\xi}{2}}^{1/m}.
\] We consider the following $\cD_n$ measurable random variable on $(\Omega,\cF)$
\begin{equation}\label{eqn:tu_def_L_n_+}
\widehat L_n
:=
\int_{[0,d_\xi^+]} h_p\crbk{\hat\pi_n(x)}\mu(dx);
\end{equation}
where the $\omega$ dependence comes only from $\hat\pi_n$. 
Since $\mu$ is the uniform distribution on $[0,1]$ and $0\leq h_p\leq 1$
\begin{equation}\label{eqn:tu_Ln_bd}
0\le \widehat L_n\le \mu\crbk{[0,d_\xi^+]}=d_\xi^+.
\end{equation}

Recall from \eqref{eqn:def_reg_hard_M} that for $\theta = \pm \xi$,
\[
\reg\crbk{\hat\pi_n;\cM_\theta^+}
=
\int_{[0,1]}\crbk{V_\theta^+(x)-Q_\theta^+\crbk{x,\hat\pi_n(x)}}\mu(dx). 
\]

Consider first the instance $\cM_{\xi}^+$. If $x\in[0,d_\xi^+]$, then
$x^m\le (d_\xi^+)^m=\xi/2$. Moreover, by \eqref{eqn:tu_A_V_breakdown}, the optimal action set for $x\in[0,d_\xi^+]$ is $\set{-1}$. Thus, 
\[
Q_{\xi}^+(x,-1)-Q_{\xi}^+(x,a)
=
(\xi-x^m)h_p(a)
\ge \frac{\xi}{2}h_p(a),
\]
for all $a\in \A$. By the non-negativity of $V_\theta^+(x) - Q_\theta^+(x,a)$, integrating over $x\in[0,d_\xi^+]$ gives
\[
\reg\crbk{\hat\pi_n;\cM_{\xi}^+} \ge
\int_{[0,d_\xi^+]}\crbk{Q_{\xi}^+(x,-1)-Q_\xi^+\crbk{x,\hat\pi_n(x)}}\mu(dx)
\ge
\frac{\xi}{2}\widehat L_n.
\]
Hence, identifying $\hat \pi_n$ with a measurable function in $m\set{\cW^n\times\cX\ra \cA }$ and inducing the measurability of $\widehat L_n$ from \eqref{eqn:tu_def_L_n_+}, we have
\begin{equation}\label{eqn:tu_+_reg_+delta}
E_\xi\sqbk{\reg\crbk{\hat\pi_n;\cM_{\xi}^+}} = \int_{\W^n}\reg\crbk{\hat\pi_n;\cM_{\xi}^+}\Psi^{n}_\xi(d\bd{w}_n)
\ge
\frac{\xi}{2}\int_{\W^n}\widehat L_n\Psi^{n}_\xi(d\bd{w}_n).
\end{equation}

Now consider the instance $\cM_{-\xi}^+$. Again by \eqref{eqn:tu_A_V_breakdown}, the optimal action
on $[0,d_\xi^+]$ is $1$, and
\[
Q_{-\xi}^+(x,1)-Q_{-\xi}^+(x,a)
=
(x^m+\xi)\crbk{1-h_p(a)}
\ge
\frac{\xi}{2}\crbk{1-h_p(a)}
\]
for all $x\in[0,d_\xi^+]$, and $a\in\A$. 
Therefore
\[
\reg\crbk{\hat\pi_n;\cM_{-\xi}^+}
\ge
\frac{\xi}{2}\int_{[0,d_\xi^+]}\crbk{1-h_p\crbk{\hat\pi_n(x)}}\,\mu(dx)
=
\frac{\xi}{2}\crbk{d_\xi^+-\widehat L_n}.
\]
We take expectations on both sides and rewrite using the law of $W_1,\ds,W_n$ to get
\begin{equation}\label{eqn:tu_+_reg_-delta}
E_{-\xi}\sqbk{\reg\crbk{\hat\pi_n;\cM_{-\xi}^+}}
\ge
\frac{\xi}{2}\crbk{d_\xi^+-\int_{\W^n}\widehat L_n\Psi^{n}_{-\xi}(d\bd{w}_n)}.
\end{equation}

Note that $\max\{a,b\}\geq \frac{1}{2}(a+b)$. Taking the maximum of \eqref{eqn:tu_+_reg_+delta} and \eqref{eqn:tu_+_reg_-delta} yields
\begin{equation}\label{eqn:tu_+_reg_tv_bd}
\begin{aligned}
\max_{\theta \in\set{\xi,-\xi}} E_{\theta}\sqbk{\reg\crbk{\hat\pi_n;\cM_{\theta}^+}}&\ge
\frac{\xi}{4}\crbk{d_\xi^+ + \int_{\W^n}\widehat L_n\Psi^{n}_\xi(d\bd{w}_n)-\int_{\W^n}\widehat L_n\Psi^{n}_{-\xi}(d\bd{w}_n)}\\
&=
\frac{\xi}{4}\crbk{d_\xi^+ - \int_{\W^n}\widehat L_n[\Psi^{n}_{-\xi}-\Psi^{n}_\xi](d\bd{w}_n)}\\
&\geq \frac{\xi}{4}d_\xi^+\crbk{1 - d_{\mrm{TV}}\crbk{\Psi^{n}_\xi,\Psi^{n}_{-\xi}}}
\end{aligned}
\end{equation}
where $d_{\mrm{TV}}$ denotes the total variation distance and the last inequality uses \eqref{eqn:tu_Ln_bd} and $\int f(d\mu-d\nu)\leq \norm f d_{\mrm{TV}}(\mu,\nu)$ for bounded function $f$ and probability measures $\mu$ and $\nu$.

We now bound the total variation distance by the Kullback–Leibler (KL) divergence. By Pinsker's inequality,
\begin{equation}\label{eqn:tu_+_tv_kl}
d_{\mrm{TV}}\crbk{\Psi^{n}_\xi,\Psi^{n}_{-\xi}}
\le
\sqrt{\frac12\KL\crbk{\Psi^{n}_\xi\middle\|\Psi^{n}_{-\xi}}}.
\end{equation}
Since under $\psi_\xi^+$ and $\psi_{-\xi}^+$ the law of $D$ is the same deterministic value
$d_\xi^+$, only the law for $Y$ differs. Hence, by the property of the KL divergence for product measures, 
\begin{align*}
\KL\crbk{\Psi^{n}_\xi\,\middle\|\,\Psi^{n}_{-\xi}}
&=
n\KL\crbk{\psi_{\xi}^+\,\middle\|\,\psi_{-\xi}^+}\\
&=
n\sqbk{\frac{1+\xi}{2}\log\frac{1+\xi}{1-\xi}+\frac{1-\xi}{2}\log\frac{1-\xi}{1+\xi}}\\
&=n\sqbk{\frac{1+\xi}{2}-\frac{1-\xi}{2}}
\log\frac{1+\xi}{1-\xi} \\
&=  n\xi\log\frac{1+\xi}{1-\xi}.   
\end{align*}

Since $\xi\le 1/4$,
\[
\log\frac{1+\xi}{1-\xi}
=
\log\crbk{1+\frac{2\xi}{1-\xi}}
\le
\frac{2\xi}{1-\xi}
\le
4\xi.
\]
Therefore,
\begin{equation}\label{eqn:tu_KL_ub}
\KL\crbk{\Psi^{n}_\xi\,\middle\|\,\Psi^{n}_{-\xi}}
\le
4n\xi^2=\frac14.
\end{equation}
Combining this with \eqref{eqn:tu_+_tv_kl}, we obtain
$d_{\mrm{TV}}\crbk{\Psi^{n}_\xi,\Psi^{n}_{-\xi}}
\le1/\sqrt 8\leq 1/2.$
Substituting this bound into \eqref{eqn:tu_+_reg_tv_bd} yields
\[
\max_{\theta \in\set{\xi,-\xi}} E_{\theta}\sqbk{\reg\crbk{\hat\pi_n;\cM_{\theta}^+}}
\ge
\frac{\xi d_\xi^+}{8}.
\]

Finally, recalling the choice $\xi = 1/(4\sqrt n)$ and the definition of $d_\xi^+$, we obtain
\[
\frac{\xi d_\xi^+}{8}
=
\frac{1}{8}\xi\crbk{\frac{\xi}{2}}^{1/m}
=
2^{-5-3/m}n^{-(m+1)/(2m)}.
\]
Therefore, for every $\hat\pi_n\in \mrm{PLA}_n$,
\[
\max_{\theta\in\Theta}
E_{\theta}\sqbk{\reg\crbk{\hat\pi_n;\cM_{\theta}^+}}
\ge\max_{\theta \in\set{\xi,-\xi}} E_{\theta}\sqbk{\reg\crbk{\hat\pi_n;\cM_{\theta}^+}}\ge
2^{-5-3/m}n^{-(m+1)/(2m)}.
\]
Taking the infimum over $\hat\pi_n\in\mrm{PLA}_n$ yields \eqref{eq:lb_margin_mass}.

\subsubsection*{Proof of Theorem \ref{thm:minimax_lb} for the $\set{\cM_\theta^-:\theta\in\Theta}$ family.}

\paragraph{Verifying Assumption \ref{assump:envelop_moment_bounds}.}
Since $\bar D_n=d_\theta^-$ a.s.$P_\theta^-$, \eqref{eqn:minimax_lb_Qhat_minus} and \eqref{eqn:Qtheta-} give
\[
\widehat Q_n^-(x,a)-Q_\theta^-(x,a)
=
(\bar Y_n-\theta)\sgn(a)\crbk{|a|^q\wedge (d_\theta^-)^q},
\]
for all $(x,a)\in \Z$ a.s.$P_\theta^-$. Recall that $d_\theta^- = |\theta|^{1/(p-q)}\le 1$. Hence,
\begin{equation}\label{eqn:tu_ub_Qtheta-_delta}
\sup_{z\in\Z}\abs{\widehat Q_n^-(z)-Q_\theta^-(z)}
\le
\abs{\bar Y_n-\theta},
\quad \text{a.s.$P_\theta^-$.}
\end{equation}

Next, fix $x\in\X$ and $a\neq b$. Then
\begin{align*}
&\frac{\abs{\crbk{\widehat Q_n^-(x,a)-Q_\theta^-(x,a)}-\crbk{\widehat Q_n^-(x,b)-Q_\theta^-(x,b)}}}{|a-b|^q}
=
\abs{\bar Y_n-\theta}
\frac{\abs{\sgn(a)\crbk{|a|^q\wedge (d_\theta^-)^q}-\sgn(b)\crbk{|b|^q\wedge (d_\theta^-)^q}}}{|a-b|^q}.
\end{align*}
If $q=0$, then the last quotient is bounded by $2$. If $q\in(0,1]$, we apply the following lemma, whose proof is deferred to Appendix \ref{section:proof:lemma:sign_power_holder}. 

\begin{lemma}\label{lemma:sign_power_holder}
Let $q\in(0,1]$. Then for all $a,b\in[-1,1]$,
\[
\bigl|\sgn(a)|a|^q-\sgn(b)|b|^q\bigr|
\le 2^{1-q}|a-b|^q.
\]
\end{lemma}

Moreover, $\sgn(a)\crbk{|a|^q\wedge (d_\theta^-)^q}$ clips $\sgn(a)|a|^q$ within the interval $\sqbk{-(d_\theta^-)^q,(d_\theta^-)^q}$ and clipping operation is $1$-Lipschitz. This, combined with Lemma \ref{lemma:sign_power_holder}, leads to
\[
\abs{\sgn(a)\crbk{|a|^q\wedge (d_\theta^-)^q}-\sgn(b)\crbk{|b|^q\wedge (d_\theta^-)^q}}
\le
2^{1-q}|a-b|^q
\le
2|a-b|^q.
\]
Consequently,
\begin{equation}\label{eqn:tu_ub_Qtheta-_Lambda}
\sup_{x\in\X}\sup_{a\neq b}
\frac{\abs{\crbk{\widehat Q_n^-(x,a)-Q_\theta^-(x,a)}-\crbk{\widehat Q_n^-(x,b)-Q_\theta^-(x,b)}}}{|a-b|^q}
\le
2\abs{\bar Y_n-\theta},
\quad \text{a.s.$P_\theta^-$.}
\end{equation}

Thus, in view of \eqref{eqn:tu_ub_Qtheta-_delta} and \eqref{eqn:tu_ub_Qtheta-_Lambda}, we define
\[
\delta_n^-:=\abs{\bar Y_n-\theta}+n^{-17}
\quad\text{and}\quad
\Lambda_n^-(q):=2\abs{\bar Y_n-\theta}+n^{-17}.
\]
Applying \eqref{eqn:tu_subgaussian_moments}, we conclude that Assumption \ref{assump:envelop_moment_bounds}
holds for the $-$ family with positive measurable envelopes $\delta_n^-$ and $\Lambda_n^-(q)$.

\paragraph{Minimax lower bound for $\set{\cM_\theta^-:\theta\in\Theta}$.}
Fix $n\ge 1$ and let $\hat\pi_n\in \mrm{PLA}_n$ be an arbitrary adapted policy learning algorithm.
As in the $+$ case, we identify $\hat\pi_n$ with a measurable map in $m\set{\cW^n\times\cX\ra \cA }$.

We override the notation and define
\[
\Psi_{\theta}^n(d\bd{w}_n):=\Psi_{\theta}^n(dw_1,\ds,dw_n):=\psi_\theta^-(dw_1)\psi_\theta^-(dw_2)\ds \psi_\theta^-(dw_n);
\]
i.e. $\Psi^n_\theta$ is the law of $W_1,\ds,W_n$ under $P_\theta^-$. 

We consider the same choice $\theta=\pm\ \xi$ where
$\xi:=1/(4\sqrt{n})$ and recall that $
d_\xi^-:=\xi^{1/(p-q)}.$
We consider the following $\cD_n$ measurable random variable on $(\Omega,\cF)$:
\begin{equation}\label{eqn:tu_def_S_n_-}
\widehat S_n
:=
\mu\crbk{\set{x\in[0,1]:\hat\pi_n(x)\ge 0}}.
\end{equation}
Since $\mu$ is a probability measure on $[0,1]$, $\widehat S_n\in[0,1]$.

Consider first the instance $\cM_{\xi}^-$. By \eqref{eqn:tu_A_V_breakdown_minus},
\[
\A_\xi^-(x)=\set{d_\xi^-}
\quad\text{and}\quad
V_\xi^-(x)=\xi(d_\xi^-)^q=(d_\xi^-)^p
\]
for $x\in[0,1]$. If $\hat\pi_n(x)<0$, then \eqref{eqn:Qtheta-} gives
\[
Q_{\xi}^-\crbk{x,\hat\pi_n(x)}
=
\xi\,\sgn(\hat\pi_n(x))\crbk{|\hat\pi_n(x)|^q\wedge (d_\xi^-)^q}
-
\big||\hat\pi_n(x)|-d_\xi^-\big|^p
\le 0.
\]
Moreover, $V_\xi^-(x)-Q_\xi^-(x,a)\ge 0$ for all $x\in\X$ and $a\in\A$. So, we must have
\[
V_{\xi}^-(x)-Q_{\xi}^-\crbk{x,\hat\pi_n(x)}
\ge
(d_\xi^-)^p\1\set{\hat\pi_n(x)<0},
\]
and integrating over $x$ yields
\begin{equation}\label{eqn:tu_-_reg_+delta}
\reg\crbk{\hat\pi_n;\cM_{\xi}^-}
\ge
(d_\xi^-)^p\mu\crbk{\set{x\in[0,1]:\hat\pi_n(x)<0}}
=
(d_\xi^-)^p(1-\widehat S_n).
\end{equation}

Now consider the instance $\cM_{-\xi}^-$. Again by \eqref{eqn:tu_A_V_breakdown_minus},
\[
\A_{-\xi}^-(x)=\set{-d_\xi^-}
\quad\text{and}\quad
V_{-\xi}^-(x)=\xi(d_\xi^-)^q=(d_\xi^-)^p,
\]
for $x\in[0,1]$. In this case, if $\hat\pi_n(x)\ge 0$, then \eqref{eqn:Qtheta-} gives
\[
Q_{-\xi}^-\crbk{x,\hat\pi_n(x)}
=
-\xi\crbk{|\hat\pi_n(x)|^q\wedge (d_\xi^-)^q}
-
\big||\hat\pi_n(x)|-d_\xi^-\big|^p
\le 0.
\]
Therefore, the same reasoning as before implies that
\[
V_{-\xi}^-(x)-Q_{-\xi}^-\crbk{x,\hat\pi_n(x)}
\ge
(d_\xi^-)^p\1\set{\hat\pi_n(x)\ge 0}.
\]
Integrating over $x$ yields
\begin{equation}\label{eqn:tu_-_reg_-delta}
\reg\crbk{\hat\pi_n;\cM_{-\xi}^-}
\ge
(d_\xi^-)^p\widehat S_n.
\end{equation}

Taking expectations using the law of $W_1,\ds ,W_n$ under $E_\xi$ and $E_{-\xi}$ respectively, we obtain
\begin{equation}\label{eqn:tu_-_reg_tv}
\begin{aligned}
\max_{\theta \in\set{\xi,-\xi}} E_{\theta}\sqbk{\reg\crbk{\hat\pi_n;\cM_{\theta}^-}}
&\ge \frac{1}{2}\crbk{E_\xi\sqbk{\reg\crbk{\hat\pi_n;\cM_{\xi}^-}}
 + E_{-\xi}\sqbk{\reg\crbk{\hat\pi_n;\cM_{-\xi}^-}}}\\
&\stackrel{(i)}{\ge} \frac{1}{2}\crbk{
\int_{\W^n}(d_\xi^-)^p(1-\widehat S_n) d\Psi_\xi^n + \int_{\W^n}(d_\xi^-)^p\widehat S_n d\Psi_{-\xi}^n }\\
&= \frac{(d_\xi^-)^p}{2}\crbk{1-
\int_{\W^n}\widehat S_n d[\Psi_\xi^n -\Psi_{-\xi}^n] }\\
&\geq
\frac{(d_\xi^-)^p}{2}\crbk{1-d_{\mrm{TV}}\crbk{\Psi^{n}_{\xi},\Psi^{n}_{-\xi}}},
\end{aligned}
\end{equation}
where $(i)$ applies \eqref{eqn:tu_-_reg_+delta} and \eqref{eqn:tu_-_reg_-delta} and the last inequality follows from $\widehat S_n\in[0,1]$. 

As in the $+$ case, we can apply Pinsker's inequality and obtain
\[
d_{\mrm{TV}}\crbk{\Psi^{n}_{\xi},\Psi^{n}_{-\xi}}
\leq \sqrt{\frac12\KL\crbk{\Psi^{n}_{\xi}\,\middle\|\,\Psi^{n}_{-\xi}}}\leq \sqrt{\frac{1}{8}}<
\frac12.
\]
where the second inequality follows from the same calculation that arrives at \eqref{eqn:tu_KL_ub}. 
Substituting this bound into \eqref{eqn:tu_-_reg_tv} yields
\[
\max_{\theta \in\set{\xi,-\xi}} E_{\theta}\sqbk{\reg\crbk{\hat\pi_n;\cM_{\theta}^-}}
\ge
\frac{(d_\xi^-)^p}{4}=
\frac14\xi^{p/(p-q)}
=
2^{-2-2p/(p-q)}n^{-p/(2(p-q))}.
\]

Therefore, for every $\hat\pi_n\in\mrm{PLA}_n$,
\[
\max_{\theta\in\Theta}
E_{\theta}\sqbk{\reg\crbk{\hat\pi_n;\cM_{\theta}^-}}
\ge
2^{-2-2p/(p-q)}n^{-p/(2(p-q))}.
\]
Taking the infimum over $\hat\pi_n\in\mrm{PLA}_n$ yields \eqref{eq:lb_growth}.
\end{proof}

\subsection{Proof of Lemma \ref{lemma:sign_power_holder}}\label{section:proof:lemma:sign_power_holder}
\begin{proof}
We split into two cases.

If $a$ and $b$ have the same sign, then
\[
\abs{\sgn(a)|a|^q-\sgn(b)|b|^q}
=
\abs{|a|^q-|b|^q}.
\]
Since $q\in(0,1]$, the map $t\ra t^q$ is $q$-H\"older with coefficient $q \le 1$ on $[0,\infty)$, so
\[
\bigl||a|^q-|b|^q\bigr|\le \bigl||a|-|b|\bigr|^q\le |a-b|^q.
\]

Now suppose $a$ and $b$ have opposite signs. Then
\[
\abs{\sgn(a)|a|^q-\sgn(b)|b|^q}
=
|a|^q+|b|^q.
\]
Because $t\ra t^q$ is concave on $[0,\infty)$,
\[
|a|^q+|b|^q\le 2^{1-q}(|a|+|b|)^q.
\]
Note that when $a$ and $b$ have opposite signs, $|a|+|b|=|a-b|$, hence
\[
\abs{\sgn(a)|a|^q-\sgn(b)|b|^q}
\le 2^{1-q}|a-b|^q.
\]

Combining the two cases proves the claim.
\end{proof}

\section{Proofs for Section \ref{section:suff_cond}}

\subsection{Proof of Theorem \ref{thm:holder_reg_of_V_Q}}\label{section:proof:thm:holder_reg_of_V_Q}
\begin{proof}
For $v\in C_b(\X)$, define the Bellman operator
\[
\cT v(x)
:=
\max_{a\in\A}
E[r(x,a,W)+\gamma v(f(x,a,W))],
\qquad x\in\X.
\] We first show that under Assumption \ref{assump:v_q_holder}, $V^*$ is the unique fixed-point of $\cT$ in $C_b(\X)$. 

By \eqref{eqn:r_weak_holder_z}, the map $z\ra E[r(z,W)]$ is continuous on $\Z$. Also, if $z_n\to z$ in $\Z$, then \eqref{eqn:f_weak_lip_z} gives
\[
E|f(z_n,W)-f(z,W)|\to 0,
\]
so $f(z_n,W)\to f(z,W)$ in probability. Hence, for every $h\in C_b(\X)$, since $h$ is bounded and continuous, $E[h(f(z_n,W))]\to E[h(f(z,W))].$
Thus, the induced transition kernel is weakly continuous. 

Therefore, by standard dynamic programming results (see \citet[Chapter 2]{hernandez2012adaptive}, the operator $\cT$ is a $\gamma$-contraction on $C_b(\X)$, and its unique fixed point is $V^*\in C_b(\X)$.

To proceed, we first show a simple inequality. Let $s,t\ge 0$ and $\theta\in[0,1]$. Then
\begin{equation}\label{eqn:pwr_bd_min}
s\wedge t \le s^\theta t^{1-\theta}.
\end{equation} 
Indeed, if $s\le t$, then $s^\theta t^{1-\theta}\ge s^\theta s^{1-\theta}=s=s\wedge t.$  If $t\le s$, then $s^\theta t^{1-\theta}\ge t^\theta t^{1-\theta}=t=s\wedge t.$ Thus in either case, $s\wedge t \le s^\theta t^{1-\theta}.$

With this result, we analyze the H\"older modulus of $\cT v$ for $v\in C_b(\X)$ with $[v]_{\alpha,\X}<\infty$. Set $z:=(x,a)\in\Z $ and $z':=(x',a)\in\Z.$ Then $|z-z'|=|x-x'|$. By \eqref{eqn:r_weak_holder_z},
\[
\abs{E[r(z,W)- r(z',W)]}
\le
E\sqbk{\abs{r(z,W)-r(z',W)}}
\le
\ell_r |x-x'|^{\alpha_r}.
\]
Also, by (i) in Assumption \ref{assump:v_q_holder}, $\abs{E[r(z,W)]-E[r(z',W)]}
\le 2r_\vee.$
Therefore,
\begin{equation}\label{eqn:tu_rwd_holder_const}
\begin{aligned}
\abs{E[r(z,W)- r(z',W)]}
&\le
\crbk{\ell_r |x-x'|^{\alpha_r}}\wedge (2r_\vee)\\
&\le
\ell_r^{\alpha/\alpha_r}(2r_\vee)^{1-\alpha/\alpha_r}|x-x'|^\alpha\\
&= (1-\gamma \ell_f^\alpha)\ell_\alpha|x-x'|^\alpha,
\end{aligned}  
\end{equation}
where we applied \eqref{eqn:pwr_bd_min} with $s=\ell_r |x-x'|^{\alpha_r}$, $t=2r_\vee$, and $\theta=\alpha/\alpha_r$. Moreover, since $\alpha\le 1$, Jensen's inequality and \eqref{eqn:f_weak_lip_z} imply
\begin{equation}\label{eqn:tu_vf_holder_const}
\begin{aligned}
E\sqbk{\abs{v(f(z,W))-v(f(z',W))}}
&\le
[v]_{\alpha,\X}
E\sqbk{\abs{f(z,W)-f(z',W)}^\alpha}\\
&\le
[v]_{\alpha,\X}
\crbk{E\sqbk{\abs{f(z,W)-f(z',W)}}}^\alpha\\
&\le
[v]_{\alpha,\X}\ell_f^\alpha |x-x'|^\alpha.
\end{aligned}
\end{equation}
Hence, combining \eqref{eqn:tu_rwd_holder_const} and \eqref{eqn:tu_vf_holder_const} yields
\[
\begin{aligned}
&\abs{
E[r(z,W)+\gamma v(f(z,W))]
-
E[r(z',W)+\gamma v(f(z',W))]
}\\
&\qquad\le
\crbk{(1-\gamma \ell_f^\alpha)\ell_\alpha+\gamma \ell_f^\alpha [v]_{\alpha,\X}}|x-x'|^\alpha.
\end{aligned}
\]

Taking the maximum over $a\in\A$ on both sides and that $|\sup_a h_1(a) - \sup_{a}h_2(a)|\leq \sup_a|h_1(a)-h_2(a)|$ for any functions $h_1,h_2$, we obtain
\[
|\cT v(x) - \cT v(x')|\leq \crbk{(1-\gamma \ell_f^\alpha)\ell_\alpha+\gamma \ell_f^\alpha [v]_{\alpha,\X}}|x-x'|^\alpha.
\]
This implies that 
\[
[\cT v]_{\alpha,\X}
\le
(1-\gamma \ell_f^\alpha)\ell_\alpha+\gamma \ell_f^\alpha [v]_{\alpha,\X}.
\]

To show the H\"older continuity of the fixed point $V^*$, we consider the Picard iteration approximation. Specifically, define $v_0\equiv 0$ and $v_{k+1}:=\cT v_k$ for $k\ge 0$. Since $[v_0]_{\alpha,\X}=0$, the preceding inequality implies by induction that $[v_k]_{\alpha,\X}\le \ell_\alpha,$ for all $k\ge 0.$

Because $\cT$ is a contraction on $C_b(\X)$, $v_k\to V^*$ uniformly. Hence, for all $x,x'\in\X$,
\[
|V^*(x)-V^*(x')|
=
\lim_{k\to\infty}|v_k(x)-v_k(x')|
\le
\ell_\alpha |x-x'|^\alpha,
\]
and therefore $[V^*]_{\alpha,\X}\le \ell_\alpha$.

Next, by the definition of $Q^*$ in \eqref{eqn:def_Qstar},
\[
Q^*(z)
=
E[r(z,W)+\gamma V^*(f(z,W))],
\qquad z\in\Z.
\]
Let $z,z'\in\Z$. Using the H\"older modulus of $V^*$ and applying the same argument as in \eqref{eqn:tu_rwd_holder_const}, we have \[
\begin{aligned}
|Q^*(z)-Q^*(z')|
&\le
\abs{E[r(z,W)- r(z',W)]}+\gamma E\sqbk{\abs{V^*(f(z,W))-V^*(f(z',W))}}\\
&\le
(1-\gamma \ell_f^\alpha)\ell_\alpha |z-z'|^\alpha
+
\gamma \ell_f^\alpha [V^*]_{\alpha,\X}|z-z'|^\alpha\\
&\le
\ell_\alpha |z-z'|^\alpha.
\end{aligned}
\]
Thus, $[Q^*]_{\alpha,\Z}\le \ell_\alpha.$ Moreover,
\[
|Q^*(z)|
\le
\abs{E[r(z,W)]}+\gamma \|V^*\|
\le
r_\vee+\gamma \|V^*\|,
\qquad z\in\Z,
\]
so $Q^*\in C_b(\Z)$.

Finally, since $\cT V^* = V^*\in C_b(\X)$ and $Q^*$ is defined by \eqref{eqn:def_Qstar}, we have that 
\[
\max_{a\in\A}Q^*(x,a) = \max_{a\in\A} E[r(x,a,W)+\gamma V^*(f(x,a,W))] = \cT V^*(x) = V^*(x)
\]
for all $x\in\X$. This completes the proof. 
\end{proof}

\subsection{Proof of Proposition \ref{prop:empirical_holder_class}}\label{section:proof:prop:empirical_holder_class}

\begin{proof}
Fix an integer $k\ge2$. By enlarging the underlying probability space if necessary, define an i.i.d. Rademacher sequence $\{\epsilon_i:i=1,\ds,n\}$ on $(\Omega,\cF,P)$ independent of $\{W_i:i=1,\ds,n\}$. Let $\Theta_0\subset\Theta$ be a countable dense set and set $\cG:=\sigma(W_1,\ds,W_n)$.

Since $\theta\to g(\theta,w)$ is $\eta$-H\"older for every fixed $w$, and since
\[
|Eg(\theta,W)-Eg(\theta',W)|
\le E\sqbk{L(W)}\,|\theta-\theta'|^\eta,
\]
the map $\theta\to (E_n-E)g(\theta,W)$ is continuous on $\Theta$. Hence
\[
Z_n=\frac1n\sup_{\theta\in\Theta_0}\abs{\sum_{i=1}^n\crbk{g(\theta,W_i)-Eg(\theta,W)}}.
\]
Therefore, by standard symmetrization theorem for empirical processes (see Theorem 3.1.21 in \citet{gine2021mathematical}),
\begin{equation}\label{eqn:emp_holder_symm}
E\sqbk{Z_n^k}
\le
\frac{2^k}{n^k}
E\sqbk{
E\sqbkcond{
\sup_{\theta\in\Theta_0}\abs{\sum_{i=1}^n \epsilon_i g(\theta,W_i)}^k
}{\cG}
}.
\end{equation}
We proceed to bound the conditional expectation within the r.h.s of \eqref{eqn:emp_holder_symm}. 

We starting with the first moment $k=1$. Define $R_{n,2}:=E_n\sqbk{R(W)^2}^{1/2}$, $L_{n,2}:=E_n\sqbk{L(W)^2}^{1/2}$. We consider the symmetrized process
\[
X_\theta:=\frac1{\sqrt n}\sum_{i=1}^n \epsilon_i g(\theta,W_i).
\] We proceed to split and bound 
\begin{equation}\label{eqn:tu_X_split}
E\sqbkcond{\sup_{\theta\in\Theta_0}\abs{X_\theta}}{\cG} \leq E\sqbkcond{\abs{X_{\theta_0}}}{\cG} + E\sqbkcond{\sup_{\theta\in\Theta_0}\abs{X_\theta-X_{\theta_0}}}{\cG}
\end{equation}
where $\theta_0\in\Theta_0$ is an arbitrary point. 

For the first term in \eqref{eqn:tu_X_split}, we observe that
$$E\sqbkcond{\abs{X_{\theta_0}}^2}{\cG}
=
E_n\sqbk{g(\theta_0,W)^2}\le R_{n,2}^2.$$
Therefore, we have $E\sqbkcond{\abs{X_{\theta_0}}}{\cG}\le R_{n,2}$. 

For the second term in \eqref{eqn:tu_X_split}, we use sub-Gaussianity of the Rademacher process. Specifically, under measure $E[\cd|\cG]$, the process $\{X_\theta:\theta\in\Theta_0\}$ is sub-Gaussian with respect to $d_n$ where
\[
d_n(\theta,\theta')
:=
E_n\sqbk{\abs{g(\theta,W)-g(\theta',W)}^2}^{1/2}
\] for $\theta,\theta'\in\Theta_0$.  This is because for every $\lambda\in\R$,
\[
\begin{aligned}
E\sqbkcond{\exp\crbk{\lambda(X_\theta-X_{\theta'})}}{\cG}
&=
\prod_{i=1}^n
E\sqbkcond{
\exp\crbk{
\frac{\lambda\epsilon_i}{\sqrt n}\crbk{g(\theta,W_i)-g(\theta',W_i)}
}
}{\cG}\\
&\le
\exp\crbk{\frac{\lambda^2}{2}d_n(\theta,\theta')^2},
\end{aligned}
\]
where we used $E_\epsilon[e^{u\epsilon}]\le e^{u^2/2}$. We note that, if we write
$d_\eta(\theta,\theta'):=|\theta-\theta'|^\eta$, then
\begin{equation}\label{eqn:emp_holder_dn}
d_n(\theta,\theta')\le L_{n,2}\,d_\eta(\theta,\theta')
\end{equation} for all $\theta,\theta'\in\Theta_0$.

Next, we apply Dudley's entropy integral bound for sub-Gaussian processes (see for example \citep[Theorem~5.22]{wainwright2019HD}) to the process
$\{X_\theta-X_{\theta_0}:\theta\in\Theta_0\}$ and $\{-(X_\theta-X_{\theta_0}):\theta\in\Theta_0\}$, yielding
\[
E\sqbkcond{\sup_{\theta\in\Theta_0}\abs{X_\theta-X_{\theta_0}}}{\cG}
\le
64\int_0^{D_n}\sqrt{\log N(u,\Theta_0,d_n)}\,du,
\]
where $D_n:=\diam(\Theta_0,d_n)$. Let
$D_\eta:=\diam(\Theta,d_\eta)$, then \eqref{eqn:emp_holder_dn} implies
$D_n\le L_{n,2}D_\eta$, and hence
\begin{align*}
\int_0^{D_n}\sqrt{\log N(u,\Theta_0,d_n)}\,du
&\stackrel{(i)}\le
\int_0^{L_{n,2}D_\eta}
\sqrt{\log N\crbk{\frac{u}{L_{n,2}},\Theta,d_\eta}}\,du \\
&\stackrel{(ii)}=
L_{n,2}\int_0^{D_\eta}\sqrt{\log N(v,\Theta,d_\eta)}\,dv \\
&=: C_{\Theta,\eta}L_{n,2}.
\end{align*}
Here, $(i)$ uses \eqref{eqn:emp_holder_dn} together with the bound
$D_n\le L_{n,2}D_\eta$, $(ii)$ is the change of variables
$v=u/L_{n,2}$, and the last line defines a finite constant that depends only on $(\Theta,\eta)$ because
$\Theta\subset\R^d$ is compact.

Combining the preceding bounds and \eqref{eqn:tu_X_split}, we obtain
\begin{equation}\label{eqn:emp_holder_first_moment}
E\sqbkcond{\sup_{\theta\in\Theta_0}\abs{\sum_{i=1}^n \epsilon_i g(\theta,W_i)} }{\cG}= \sqrt{n}E\sqbkcond{\sup_{\theta\in\Theta_0}\abs{X_\theta}}{\cG}
\le
\sqrt{n}\crbk{R_{n,2}+C_{\Theta,\eta}L_{n,2}}.
\end{equation}

To generalize this to integers $k\ge2$, we use the Khinchin-Kahane inequality. Specifically, consider the following countable set
\[
T_n:=
\Bigl\{
\crbk{g(\theta,W_1),\ds,g(\theta,W_n)}:\theta\in\Theta_0
\Bigr\}
\cup
\Bigl\{
-\crbk{g(\theta,W_1),\ds,g(\theta,W_n)}:\theta\in\Theta_0
\Bigr\}
\subset\R^n.
\]
Then
\[
\sup_{\theta\in\Theta_0}\abs{\sum_{i=1}^n \epsilon_i g(\theta,W_i)}
=
\sup_{t\in T_n}\sum_{i=1}^n t_i\epsilon_i.
\]
Therefore, Proposition~3.2.8 of \citet{gine2021mathematical} with $p=k$ and $q=1$ gives
\[
E\sqbkcond{
\sup_{\theta\in\Theta_0}\abs{\sum_{i=1}^n \epsilon_i g(\theta,W_i)}^k
}{\cG}^{1/k}
\le
C_k\,
E\sqbkcond{
\sup_{\theta\in\Theta_0}\abs{\sum_{i=1}^n \epsilon_i g(\theta,W_i)}
}{\cG},
\]
where $C_k<\infty$ depends only on $k$. Thus, it follows from \eqref{eqn:emp_holder_first_moment} that
\[
E\sqbkcond{
\sup_{\theta\in\Theta_0}\abs{\sum_{i=1}^n \epsilon_i g(\theta,W_i)}^k
}{\cG}
\le
C_k^k n^{k/2}\crbk{R_{n,2}+C_{\Theta,\eta}L_{n,2}}^k.
\]
Substituting this estimate into \eqref{eqn:emp_holder_symm}, we conclude that for all $k\geq 2$
\begin{equation}\label{eqn:tu_EZnk_temp}
E\sqbk{Z_n^k}
\le
\frac{C_k'}{n^{k/2}}
E\sqbk{\crbk{R_{n,2}+C_{\Theta,\eta}L_{n,2}}^k}
\le
\frac{C_k''}{n^{k/2}}
\crbk{
E\sqbk{R_{n,2}^k}+E\sqbk{L_{n,2}^k}
}
\end{equation}
for some $C''_k < \infty$ depending on $k$ and $C_{\Theta,\eta}$, where we used $(a+b)^k\le 2^{k-1}(a^k+b^k)$.

Finally, to finish the proof, we observe that since $k/2\ge1$, Jensen's inequality yields
\[
R_{n,2}^k
=
\crbk{E_n\sqbk{R(W)^2}}^{k/2}
\le
E_n\sqbk{R(W)^k}
\quad\text{and}\quad
L_{n,2}^k
=
\crbk{E_n\sqbk{L(W)^2}}^{k/2}
\le
E_n\sqbk{L(W)^k}.
\]
Taking expectations, we obtain
\[
E\sqbk{R_{n,2}^k}\le E\sqbk{R(W)^k}\quad\text{and}\quad
E\sqbk{L_{n,2}^k}\le E\sqbk{L(W)^k}.
\]
Combining this with \eqref{eqn:tu_EZnk_temp} completes the proof of Proposition~\ref{prop:empirical_holder_class}.
\end{proof}

\subsection{Proof of Theorem \ref{thm:split_sample_est}}\label{section:proof:thm:split_sample_est}

\begin{proof}
We divide the proof into two parts. First, we establish the H\"older continuity statements for $Q^*$ and $\widehat Q_n$, which in turn verify the properties in Assumption~\ref{assump:regularity}. We then use these H\"older continuity together with Proposition~\ref{prop:empirical_holder_class} to show that the envelope moment bounds in Assumption~\ref{assump:envelop_moment_bounds} are satisfied.
\paragraph{H\"older continuity.}  Fix $\alpha\in(0,\alpha_r]$ such that $\gamma\ell_f^\alpha<1$, and let $z_0$ be the reference point from Assumption \ref{assump:suff_cond_A1A2} part (iii). Define
$$R(w):=|r(z_0,w)|+\diam(\Z)^{\alpha_r}L_r(w),$$ and let $r_\vee:=E[R(W)]$ and $
\ell_r:=E[L_r(W)]$.
Since $\Z$ is compact, by Assumption \ref{assump:suff_cond_A1A2} (iii) and (iv), $R(W)$ has moments of every order. Moreover, note that
\[
\begin{aligned}
|r(z,w)|&\le |r(z_0,w)|+|r(z,w)-r(z_0,w)|\\
&\leq|r(z_0,w)| + L_r(w)|z-z_0|^{\alpha_r}\\
&\le R(w)   
\end{aligned}
\]
for all $ z\in\Z$ and $\ w\in\W.$ 

Next, as in Theorem \ref{thm:holder_reg_of_V_Q}, we set
\[
\ell_\alpha:=\frac{\ell_r^{\alpha/\alpha_r}(2r_\vee)^{1-\alpha/\alpha_r}}{1-\gamma\ell_f^\alpha}.
\]
For $\mu\in\{\psi,\psi_n^V,\psi_n^Q\}$, write $E_\mu\phi(W):=\int_\W \phi(w)\mu(dw)$. The preceding bound together with Assumption \ref{assump:suff_cond_A1A2} (iv) and (v) yields
\[
\begin{aligned}
E_\mu|r(z,W)|&\le E_\mu R(W),\\
E_\mu|r(z,W)-r(z',W)|&\le E_\mu L_r(W)\,|z-z'|^{\alpha_r},\\
E_\mu|f(z,W)-f(z',W)|&\le \ell_f|z-z'|.
\end{aligned}
\]
Therefore, Theorem \ref{thm:holder_reg_of_V_Q} applies to the case $\mu = \psi$, while Corollary \ref{cor:Q_Qn_regular} applies to measures $\psi_n^V(\omega,\cd)$ and $\psi_n^Q(\omega,\cd)$. In particular, $Q^*$ is $\alpha$-H\"older on $\Z$, $[V^*]_{\alpha,\X}\le \ell_\alpha$, and
\begin{equation}\label{eqn:tu_hatV_holder_modulus}
\sqbk{\widehat V_n}_{\alpha,\X}\le \widehat L_{\alpha,n}^V
:=
\frac{\crbk{E_n^V L_r(W)}^{\alpha/\alpha_r}\crbk{2E_n^V R(W)}^{1-\alpha/\alpha_r}}{1-\gamma\ell_f^\alpha}.
\end{equation}

To show the H\"older continuity of $\widehat Q_n$, we start from the bounds
\[
|r(z,w)-r(z',w)|\le L_r(w)|z-z'|^{\alpha_r}
\quad\text{and}\quad
|r(z,w)-r(z',w)|\le 2R(w). 
\]
So, inequality \eqref{eqn:pwr_bd_min} with $\theta=\alpha/\alpha_r$, implies 
\begin{equation}\label{eqn:tu_r_holder}
|r(z,w)-r(z',w)|\le \kappa_r(w)|z-z'|^\alpha
\end{equation}
for all $z,z'\in\Z$, where
\[
\kappa_r(w):=L_r(w)^{\alpha/\alpha_r}(2R(w))^{1-\alpha/\alpha_r}.
\]

Using \eqref{eqn:def_hatQn}, the bound on $[\widehat V_n]_{\alpha,\X}$, and Assumption \ref{assump:suff_cond_A1A2} (v), we obtain for all $z,z'\in\Z$,
\[
\begin{aligned}
|\widehat Q_n(z)-\widehat Q_n(z')|
&\le
E_n^Q|r(z,W)-r(z',W)| +
\gamma E_n^Q\abs{\widehat V_n(f(z,W))-\widehat V_n(f(z',W))} \\
&\le
\crbk{E_n^Q\kappa_r(W)+\gamma \ell_f^\alpha \widehat L_{\alpha,n}}|z-z'|^\alpha.
\end{aligned}
\]
Thus $\widehat Q_n(\omega,\cd)$ is $\alpha$-H\"older on $\Z$ for every $\omega$ and the first claim in Theorem \ref{thm:split_sample_est} follows. In particular, Assumption \ref{assump:regularity} holds.

\paragraph{Envelopes and moment bounds.} We now verify Assumption \ref{assump:envelop_moment_bounds}. We fix $ q\in[0,\bar q]$ and $q<\alpha^*$. Then, by the definition of $\alpha^*$, there exists $\alpha\in( q,\alpha_r]$ such that $\gamma\ell_f^\alpha<1$. Since $R(W)$ and $L_r(W)$ have moments of every order, Jensen's inequality gives
$$E\sqbk{\crbk{E_n^V L_r(W)}^p}\leq EE_n^V [L_r(W)^p] = E[L_r(W)^p]$$ and similarly $E\sqbk{\crbk{E_n^V R(W)}^p} \leq E[R(W)^p]$
for all $p\ge 1$. Therefore, for every integer $k\ge1$,
\begin{equation}\label{eqn:tu_Ln_moment_bd}
\sup_{n\ge1}E\sqbk{\crbk{\widehat L_{\alpha,n}^V}^k}=:L_k<\infty.
\end{equation}

Next, define
\[
g(z,w):=r(z,w)+\gamma V^*(f(z,w)).
\]
Then $|g(z,w)|\le R(w)+\gamma\|V^*\|_\infty$. Also, since $[V^*]_{\alpha,\X}\le \ell_\alpha$ and \eqref{eqn:tu_r_holder} holds,
\begin{align*}
|g(z,w)-g(z',w)|
&\le \kappa_r(w)|z-z'|^\alpha + \gamma \ell_\alpha |f(z,w)-f(z',w)|^\alpha \\
&\leq
\crbk{\kappa_r(w)+\gamma\ell_\alpha\ell_f^\alpha}|z-z'|^\alpha   
\end{align*}
for all $z,z'\in\Z$. Hence with $$Z_n^U:=\sup_{z\in\Z}|(E_n^U-E)g(z,W)|$$ where $U = Q$ or $V$, 
Proposition \ref{prop:empirical_holder_class} applies, yielding
\begin{equation}\label{eqn:tu_Zk_moment_bd}
E\sqbk{\crbk{Z_n^U}^k}\le \frac{C_k}{n^{k/2}},
\end{equation}
for all integer $k\ge2$ and $U = Q$ or $V$.

With this moment bound, we verify the first part of Assumption \ref{assump:envelop_moment_bounds}. From the definition of $\widehat Q_n$,
\begin{equation}\label{eqn:sample_split_q_decomp}
\begin{aligned}
\sup_{z\in \Z}\abs{\widehat Q_n(z)-Q^*(z)}
&=
\sup_{z\in \Z}\abs{(E_n^Q-E)g(z,W)+\gamma E_n^Q\sqbk{(\widehat V_n-V^*)(f(z,W))}}\\
&\leq Z_n^Q + \gamma\sup_{z\in\Z}\abs{E_n^Q\sqbk{(\widehat V_n-V^*)(f(z,W))}}\\
&\leq Z_n^Q + \gamma \norm{\widehat V_n-V^*}
\end{aligned}
\end{equation}
Then, to achieve the first envelope moment bounds in Assumption \ref{assump:envelop_moment_bounds}, it remains to bound the moments of $\norm{\widehat V_n-V^*}$. By the Bellman equation for $V^*$ and the definition of $\widehat V_n$ in \eqref{eqn:def_hatVn}, we see that
$$\begin{aligned}\norm{\widehat V_n-V^*}&\leq \sup_{z\in\Z}\abs{(E_n^V -E)r(z,W) + \gamma E_n^V\sqbk{\widehat V_n(f(z,W))} -  \gamma E\sqbk{V^*(f(z,W))}}\\
& \leq \sup_{z\in\Z} \abs{(E_n^V -E)g(z,W)} + \gamma\sup_{z\in\Z} \abs{E_n^V\sqbk{\widehat V_n(f(z,W)) - V^*(f(z,W))}}\\
&\leq Z_n^V + \gamma \norm{\widehat V_n-V^*}
\end{aligned}$$
Therefore, rearranging the inequality, we get
\begin{equation}\label{eqn:sample_split_delta}
\norm{\widehat V_n-V^*}
\le \frac{Z_n^V}{1-\gamma}.
\end{equation}
From \eqref{eqn:sample_split_q_decomp} and \eqref{eqn:sample_split_delta}, we see that
\[
\sup_{z\in\Z}|\widehat Q_n(z)-Q^*(z)|
\le \delta_n:= Z_n ^Q+ \frac{\gamma Z_n^V}{1-\gamma} + n^{-1}
\]
with measurable $\delta_n > 0$, where the measurability follows from the continuity of $g$ and the compactness of $\Z$. Moreover, applying \eqref{eqn:tu_Zk_moment_bd} with $k=2$,  we conclude that $E\sqbk{\delta_n^2}\le C/n$ for some $C < \infty$.

Next, we verify the divided-difference bound in Assumption \ref{assump:envelop_moment_bounds}. Define
\[
G_n^V(z,w):=r(z,w)+\gamma\widehat V_n(f(z,w)),
\qquad
B_n^V(w):=\kappa_r(w)+\gamma\ell_f^\alpha\widehat L_{\alpha,n}^V.
\]
Then, by \eqref{eqn:tu_r_holder}, \eqref{eqn:tu_hatV_holder_modulus}, and the Lipschitz property of $f$,
\begin{equation}\label{eqn:sample_split_Gn_holder}
|G_n^V(z,w)-G_n^V(z',w)|\le B_n^V(w)|z-z'|^\alpha
\end{equation}
for all $z,z'\in\Z$ and $ w\in\W$. Moreover, notice that
\[
\widehat Q_n(z)=E_n^Q\sqbk{G_n^V(z,W)}\quad\text{and}\quad
Q^*(z)=E\sqbk{G_n^V(z,W)}-\gamma\int_\X (\widehat V_n-V^*)(y)\,P(dy|z).
\]

We subtract the identities and divide by $|a-b|^{ q}$ to get
\begin{equation}\label{eqn:sample_split_divdiff}
\begin{aligned}
&\sup_{x\in\X,a\neq b}\frac{\abs{\crbk{\widehat Q_n(x,a)-Q^*(x,a)}-\crbk{\widehat Q_n(x,b)-Q^*(x,b)}}}{|a-b|^{q}}\\
& \le
\underbrace{\sup_{x\in\X,a\neq b }
\abs{(E_n^Q-E)\sqbk{\frac{G_n^V(x,a,W)-G_n^V(x,b,W)}{|a-b|^{q}}}}}_{=:Y_n}\\
&\quad + 
\underbrace{\gamma\sup_{x\in\X,a\neq b}
\frac{\abs{\int_\X (\widehat V_n-V^*)(y)\,\crbk{P(dy|x,a)-P(dy|x,b)}}}{|a-b|^{q}}}_{=:U_n^V},
\end{aligned}
\end{equation}

We bound the two terms separately. We control $U_n^V$ using by Assumption \ref{assump:suff_cond_A1A2} (ii). Specifically, since $\bar q\geq q$ and $\A$ is compact,
\begin{equation}\label{eqn:sample_split_transition_term}
\begin{aligned}
U^V_n&\leq \gamma\norm{\widehat V_n - V^*}\sup_{x\in\X,a\neq b}
\frac{d_{\mrm{TV}}(P(\cd|x,a),P(\cd|x,b))}{|a-b|^{ q}}\\
&\le
2\gamma \ell_P\norm{\widehat V_n - V^*} \sup_{a\neq b}|a-b|^{\bar q - q}\\
&=2\gamma \ell_P\norm{\widehat V_n - V^*} \diam(\A)^{\bar q -q}
\end{aligned}
\end{equation}

To bound $Y_n$, let $\Theta:=\X\times\A\times\A$. For $\theta=(x,a,b)\in\Theta$, define
\[
H_n^V(\theta,w):=
\begin{cases}
\dfrac{G_n^V(x,a,w)-G_n^V(x,b,w)}{|a-b|^{q}}, & a\ne b,\\[1ex]
0, & a=b.
\end{cases}
\]
We claim that the following lemma holds for $H_n^V$. The proof of Lemma \ref{lemma:H_holder_bd} is deferred to Appendix \ref{section:proof:lemma:H_holder_bd}.

\begin{lemma}\label{lemma:H_holder_bd}
There exists a constant $C_H<\infty$, depending only on $(\alpha, q,\A)$, such that
\begin{equation}\label{eqn:tu_H_holder}
|H_n^V(\theta,w)-H_n^V(\theta',w)|\leq
C_H B_n^V(w)|\theta-\theta'|^{\alpha- q},
\end{equation}
for all $\theta,\theta'\in\Theta$, and $w\in\W$.
\end{lemma}

Now we condition on $\mathcal F_n^V:=\sigma(W_1^V,\dots,W_n^V)$. Then, $H_n^V$ and $B_n^V$ are deterministic. Note that, by Lemma \ref{lemma:H_holder_bd}, $\theta\ra H_n^V(\theta,w)$ is H\"older continuous in the sense of Proposition \ref{prop:empirical_holder_class}. Moreover, by \eqref{eqn:sample_split_Gn_holder} and Lemma \ref{lemma:H_holder_bd} both $|H_n^V(\cd,w)|$ and the H\"older coefficient $\sqbk{H_n^V(\cd,w)}_{\alpha- q,\Theta}$ are bounded by $(\diam(\A)^{\alpha- q}\vee C_H)B_n^V(w)$. Since
\[
B_n^V(w):= \kappa_r(w)+\gamma\ell_f^\alpha\widehat L_{\alpha,n}^V,
\]
and $\kappa_r(W)$ has moments of every order, we have that for all integers $k\ge2$,
\[
E\sqbkcond{B_n^V(W)^k}{\mathcal F_n^V}\le C_k\crbk{1+\widehat L_{\alpha,n}^k},
\]
for some $C_k$ that only depend on $(k,\A,\alpha, q,\ell_\alpha,\gamma,E[\kappa_r(W)^k])$.

Hence, under $E[\cd|\mathcal F_n^V]$, Proposition \ref{prop:empirical_holder_class} applies to the class $\{H_n^V(\theta,\cdot):\theta\in\Theta\}$ on the compact parameter space $\Theta$. 
Therefore, we conclude that for every integer $k\ge2$,
\[
E\sqbkcond{Y_n^k}{\mathcal F_n^V}
\le
\frac{C_k\crbk{1+\widehat L_{\alpha,n}^k}}{n^{k/2}},
\]
Taking expectations and using \eqref{eqn:tu_Ln_moment_bd}, we obtain
\begin{equation}\label{eqn:sample_split_Yn_moment}
E\sqbk{Y_n^k}\le \frac{C_k}{n^{k/2}}
\end{equation}
for some possibly different $C_k$.

Combining \eqref{eqn:sample_split_divdiff}, \eqref{eqn:sample_split_transition_term}, we conclude that
\[
\sup_{x\in\X}\sup_{a\ne b}
\frac{\abs{\crbk{\widehat Q_n(x,a)-Q^*(x,a)}-\crbk{\widehat Q_n(x,b)-Q^*(x,b)}}}{|a-b|^{ q}}
\le
\Lambda_n( q),
\]
where
\[
\Lambda_n( q):=
Y_n+2\gamma\ell_P\diam(\A)^{\bar q-q}\norm{\widehat V_n - V^*}+n^{-1}.
\]
Finally, by \eqref{eqn:sample_split_Yn_moment}, \eqref{eqn:tu_Zk_moment_bd}, and \eqref{eqn:sample_split_delta}, we conclude that for all integers $k\ge2$
\[
E\sqbk{\Lambda_n( q)^k}\le \frac{C_k}{n^{k/2}},
\]
Hence, the second part of Assumption \ref{assump:envelop_moment_bounds} follows.

\end{proof}

\subsubsection{Proof of Lemma \ref{lemma:H_holder_bd}}\label{section:proof:lemma:H_holder_bd}

\begin{proof}
Let $p:=\alpha- q\in(0,1]$. To show \eqref{eqn:tu_H_holder}, fix $w\in\W$ and $\theta=(x,a,b)$, $\theta'=(x',a',b')\in\Theta$. Write
\[
d:=|\theta-\theta'|,\qquad s:=|a-b|,\qquad t:=|a'-b'|,
\]
and assume without loss of generality that $s\ge t$.

If $q=0$, then \eqref{eqn:sample_split_Gn_holder} gives
\[
\begin{aligned}
|H_n^V(\theta,w)-H_n^V(\theta',w)|
&\le |G_n^V(x,a,w)-G_n^V(x',a',w)|+|G_n^V(x,b,w)-G_n^V(x',b',w)| \\
&\le 2B_n^V(w)d^\alpha,
\end{aligned}
\]
which is \eqref{eqn:tu_H_holder} with $C_H = 2$, as $p=\alpha$.

Assume now that $q>0$. By \eqref{eqn:sample_split_Gn_holder},
\begin{equation}\label{eqn:tu_H_pointwise}
|H_n^V(\theta,w)|\le B_n^V(w)s^p,
\qquad
|H_n^V(\theta',w)|\le B_n^V(w)t^p.
\end{equation}
If $t=0$, then $H_n^V(\theta',w)=0$, so \eqref{eqn:tu_H_pointwise} yields
\[
\begin{aligned}
|H_n^V(\theta,w)-H_n^V(\theta',w)|
=
|H_n^V(\theta,w)|
\le B_n^V(w)s^p. 
\end{aligned}
\]
Since $|u|-|v| \leq |u-v|$, we have \begin{equation}\label{eqn:tu_d_ub_st}
   s = s-t\leq |a-a'+b'-b|\leq |a-a'| + |b-b'|\leq 2d. 
\end{equation} So, $$|H_n^V(\theta,w)-H_n^V(\theta',w)| \leq  2^pB_n^V(w)d^p,$$
showing \eqref{eqn:tu_H_holder} with $C_H = 2^p$.

Hence it remains to consider the case $s\ge t>0$ and $q>0$. Set
\[
D:=G_n^V(x,a,w)-G_n^V(x,b,w),
\qquad
D':=G_n^V(x',a',w)-G_n^V(x',b',w).
\]
Then, it is easy to see that
\[
|H_n^V(\theta,w)-H_n^V(\theta',w)|
\le \underbrace{ s^{-q}|D-D'|}_{=:I_1}+\underbrace{|D'|\abs{s^{-q}-t^{-q}}}_{=:I_2}.
\]
We then bound $I_1$ and $I_2$ separately. 

For $I_1$, \eqref{eqn:sample_split_Gn_holder} implies $|D-D'|\le 2B_n^V(w)d^\alpha$, while \eqref{eqn:tu_H_pointwise} implies
\[
|D-D'|\le |D|+|D'|\le B_n^V(w)(s^\alpha+t^\alpha)\le 2B_n^V(w)s^\alpha.
\]
Hence, we have
\begin{equation}\label{eqn:tu_I1_bd}
I_1\le 2B_n^V(w)\min\{d^\alpha s^{-q},\,s^p\}\le 2B_n^V(w)d^p.
\end{equation}
Here, the last inequality is from the following two cases: if $s\le d$, then $s^p\le d^p$, whereas if $s>d$, then $d^\alpha s^{-q}\le d^{\alpha-q}=d^p$.

For $I_2$, using $|D'|\le B_n^V(w)t^\alpha$ and writing $\lambda:=t/s\in(0,1]$, we obtain
\[
I_2\le B_n^V(w)t^\alpha(t^{-q}-s^{-q})
= B_n^V(w)s^p(\lambda^p-\lambda^\alpha).
\]
We claim that for $\lambda\in [0,1]$, \begin{equation}\label{eqn:tu_lambda_bd}
    \lambda^p-\lambda^\alpha\leq C(1-\lambda)^p
\end{equation} for some $C < \infty$ that only depend on $(\alpha,p)$. Indeed, consider
\[
\phi(\lambda):=
\frac{\lambda^p-\lambda^\alpha}{(1-\lambda)^p}
\]
and note that $\phi$ is continuous on $[0,1)$. Therefore, to show \eqref{eqn:tu_lambda_bd}, it suffices to check that $\lim_{\lambda\ua 1} \phi(\lambda) < \infty$.  By L'H\^ospital's rule, 
$$\lim_{\lambda\ua 1} \phi(\lambda) = \lim_{\lambda\ua 1}\frac{p\lambda^{p-1} - \alpha\lambda^{\alpha-1}}{-p(1-\lambda)^{p-1}} = \lim_{\lambda\ua 1}\frac{q}{p(1-\lambda)^{p-1}}\leq \frac{q}{p}.$$
Hence, \eqref{eqn:tu_lambda_bd} holds and
\[
I_2\le CB_n^V(w)s^p(1-\lambda)^p = CB_n^V(w)(s-t)^p
\le C2^p B_n^V(w)d^p,
\]
where the last inequality follows from \eqref{eqn:tu_d_ub_st}. 

Combining the bounds for $I_1$ and $I_2$, we conclude that
\[
|H_n^V(\theta,w)-H_n^V(\theta',w)|
\le C_H B_n^V(w)d^{\alpha-q}
\]
for some constant $C_H<\infty$ depending only on $(\alpha,p = \alpha-q)$. This proves \eqref{eqn:tu_H_holder}.
\end{proof}

\end{document}